\documentclass[12pt, A4]{article}
\usepackage[utf8]{inputenc}
\usepackage{mathrsfs}
\usepackage{graphicx}
\usepackage{amsfonts}
\usepackage{amsmath}
\usepackage{bm}
\usepackage{bbm}
\usepackage{amssymb}
\usepackage{amsthm}
\usepackage{xcolor}
\usepackage{multirow}
\usepackage{natbib}
\usepackage[section]{placeins}
\newtheorem{theorem}{Theorem}
\newtheorem{lemma}{Lemma}

\newcommand{\SumOverN}{\sum_{i=1}^n}
\newcommand{\ObsRes}{\textbf{y}_i}
\newcommand{\ObsCov}{\textbf{x}_i}
\newcommand{\RegCoe}{\textbf{b}}
\newcommand{\LogCoe}{\boldsymbol{\gamma}}
\newcommand{\vect}{\emph{vec}}
\newcommand{\IdenMat}{\textbf{I}}
\newcommand{\Cov}{\mathbb{C}\text{ov}}

\newcommand{\Norm}[1]{ \left\| #1 \right\| }

\newcommand{\Parentheses}[1]{ \left( #1 \right) }

\newcommand{\Curlybracket}[1]{ \left\{ #1 \right\} }

\newcommand{\Span}{\text{span}}
\usepackage{geometry}
\newgeometry{vmargin={20mm}, hmargin={24mm,27mm}}
\title{Doubly robust estimation with functional outcomes missing at random}
\author{Xijia Liu\footnote{Correponding author: xijia.liu@umu.se}, Kreske Ecker, Lina Schelin and Xavier de Luna \\Department of Statistics/USBE, Ume\aa\ University, Ume\aa\ , Sweden}
\date{}
\begin{document}
\maketitle
\begin{abstract}
    We present and study semi-parametric estimators for the mean of functional outcomes in situations where some of these outcomes are missing, and covariate information is available on all units. 
    Assuming that the missingness mechanism depends only on the covariates (missing at random assumption), we present two estimators for the functional mean parameter, using working models for the functional outcome given the covariates, and the probability of missingness given the covariates. 
    We contribute by establishing that both these estimators have Gaussian processes as limiting distributions and explicitly give their covariance functions. 
    One of the estimators is doubly robust in the sense that the limiting distribution holds whenever at least one of the nuisance models is correctly specified. 
    These results allow us to present simultaneous confidence bands for the mean function with asymptotically guaranteed coverage. 
    A Monte Carlo study shows the finite sample properties of the proposed functional estimators and their associated simultaneous inference. 
    The use of the method is illustrated in an application where the mean of counterfactual outcomes is targeted.
\end{abstract}

\section{Introduction}\label{sec1}
        We propose and study semi-parametric estimators of the mean of functional outcomes when these functions are not observed for all individuals in the study. 
        Functional outcomes arise in many different situations, e.g. when studying human movement \citep{dannenmaier2020application} or lifetime cumulative income trajectories \citep{ecker2022regional}, \citep[see also][for further examples]{ullah2013applications}.
        Missingness of the outcomes can occur due to e.g., non-response (in a survey context) or drop-out (for follow-up studies). 
        Missingness can also be a consequence of considering the potential outcome framework for causal inference allowing us to answer counterfactual questions; see \cite{holland86}, \cite{missingcausal}, and Section \ref{sec:app} for an application. 
        We focus here on the case where some baseline covariates are observed for all units, and the missingness mechanism is governed only by these covariates. This is referred to as the missing at random assumption \cite[see e.g.,][]{tsiatis2006semiparametric, little2019statistical}. 
        Because the entire functional outcome is missing for some units, our context is different from the literature on fragmented functional data, where it is assumed that functional objects may only be observed on some subset(s) of their domain. 
        This literature does typically not consider access to covariates and instead aims to reconstruct the missing sections using only information from the observed functional outcomes \cite{delaigle2016approximating, kraus2019inferential, kneip2020optimal}.

        There is a large body of literature on semi-parametric estimators of the mean when a scalar outcome is missing at random \cite[e.g.,][]{tsiatis2006semiparametric, little2019statistical, molenberghs2015handbook}. 
        Popular estimators include outcome regression (also called regression imputation) estimators \cite[e.g.,][]{Tan,genback2019causal} and augmented inverse probability weighted estimators \cite[also called doubly robust estimators; e.g.,][]{robins1994estimation, bang2005doubly, kang2007demystifying}. 
        The outcome regression (OR) method involves modelling the relationship between the outcome and the observed covariates and imputing the missing outcomes with predictions from this model. 
        The doubly robust (DR) estimator can be interpreted as a correction of the OR estimator, using weights that are based on modelling the missingness mechanism as a function of the covariates. 
        The DR estimator is locally efficient (it attains the semi-parametric efficiency bound if both models are correct) and doubly robust (consistent and asymptotically normal if only one of the two models is correct;  \cite[e.g.,][]{Seaman:2018}).

        In this paper, we contribute by studying the OR and DR estimators in the context of functional outcomes. 
        We provide their limiting distributions by considering first a discretisation of the functional domain, and showing that the resulting multivariate, but finite-dimensional, estimator is asymptotically multivariate normal. 
        By then showing the tightness of the functional estimator, we obtain a Gaussian process as the limiting distribution of the (infinite dimensional) functional estimator of the mean. 
        The double robustness property holds for the functional DR estimator, where the limiting process is Gaussian and $\sqrt{n}$-consistent even when only one of the nuisance models is correctly specified. 
        
        These asymptotic properties allow us to use recent results in \cite{Liebl} to propose simultaneous confidence bands for the functional mean, providing a given coverage level over the entire domain. 
        At the same time, these bands also allow for local inferences and the interpretation of local significance levels over user-determined partitions of the domain. 
        Up to our knowledge, there are only few earlier peer-reviewed publications on doubly robust estimators in the present context. 
        \cite{Belloni}, in the causal inference context, study doubly robust estimators and obtain their limiting distribution, but without an expression for the covariance function. 
        Instead, they propose a resampling scheme to obtain valid inference. 
        In another related contribution, \cite{ecker2023causal} study an OR estimator for a functional causal effect and provides its finite sample distribution, albeit under stricter distributional assumptions about the underlying data. 
        
        The paper is organised as follows. 
        Section 2 presents the estimators and their asymptotic properties. 
        In Section 3, a Monte Carlo study illustrates the finite sample properties of the estimators, and in particular shows that the simultaneous confidence bands have the expected empirical coverages in a range of situations. 
        Section 4 presents an application where we study a counterfactual question: what would be the average life income trajectory for the Swedish cohort born in 1954 if all had been living at age 20 in a large labour market region? A discussion concludes the paper in Section 5.

    \section{Theory and method}
        \subsection{Model and semi-parametric estimation}
        Suppose we have a random sample, $(\mathcal{Y}_i, \ObsCov, Z_i)_{i=1}^n$, from the joint distribution of $(\mathcal{Y}, \textbf{x}, Z)$, where the outcome variable, $\mathcal{Y}$, is a $L^2([0,1], \mathbb{R})$ valued random element, and the set of covariates, $\textbf{x}$, is a $p \times 1$ random vector with mean vector $\boldsymbol{\mu}_x$ and covariance matrix $\boldsymbol{\Sigma}_x$. 
        Further, $Z$ is a binary variable indicating whether the outcome $\mathcal{Y}$ is observed, in which case $Z=1$, or missing ($Z=0$). 
        We assume throughout that $\Pr[ Z_i=1 | \ObsCov, \mathcal{Y}_i] = \Pr[ Z_i=1 | \ObsCov ]$, that is, the missingness mechanism is ignorable (also referred to as the outcome being missing at random, \cite[][Chap 6]{tsiatis2006semiparametric}). 
        The parameter of interest is the functional mean $\mu_{y} = \mathbb{E}\Parentheses{\mathcal{Y}}$.

        For the outcome variable, we specify a working model as a multiple-functional regression:
        \begin{equation}\label{equ:regmod1}
            \mathcal{Y}_i = \ObsCov^{\top}\boldsymbol{\beta} + \epsilon_i,
        \end{equation}
        where $\boldsymbol{\beta} = (\beta_1, \dots, \beta_p)$ and $\beta_j$s are non-stochastic functions in $L^2([0,1], \mathbb{R})$, and $\epsilon_i$ is a $L^2([0,1], \mathbb{R})$ valued random element with zero mean and with $\sigma_{\epsilon}$ as its covariance function.    
        Furthermore, a working model for $\Pr[ Z_i=1 | \ObsCov ]$ is specified as the logistic regression with the conditional probability of being observed in the sample:
            \begin{equation}\label{equ:ps}
               \tau(\ObsCov^{\top} \LogCoe),
            \end{equation}
        where $\tau(s) = (1+e^{-s})^{-1}$. 
        The outcome regression (OR) imputed estimator of the expected value of the functional outcome variables, $\mu_{y}$, based on model \eqref{equ:regmod1} is defined as 
            \begin{align}\label{equ:OR}
                \hat{\mu}_{OR} = n^{-1}\SumOverN \ObsCov^{\top}\hat{\boldsymbol{\beta}},
            \end{align}
        where $\hat{\boldsymbol{\beta}}$ is the functional estimation
            $
                \Parentheses{ \mathbf{X}^{\top}\mathbf{X} }^{-1} \mathbf{X}^{\top}\boldsymbol{\mathcal{Y}},
            $
        with 
        $\boldsymbol{\mathcal{Y}} = \Parentheses{\mathcal{Y}_1, \dots, \mathcal{Y}_n}^{\top}$, and $\mathbf{X} = \Parentheses{Z_1\mathbf{x}_1, \dots, Z_n\mathbf{x}_n}^{\top}$.
        Based on the two working models \eqref{equ:regmod1} and \eqref{equ:ps}, the doubly robust estimator of ${\mu}_{y}$, is defined as
            \begin{align}\label{equ:DR}
                \hat{\mu}_{DR} = n^{-1} \SumOverN \left( \ObsCov^{\top}\hat{\boldsymbol{\beta}} + \frac{Z_i}{\tau(\ObsCov^{\top}\hat{\LogCoe})}\Parentheses{\mathcal{Y}_i -\ObsCov^{\top}\hat{\boldsymbol{\beta}}}  \right),
            \end{align}
        where $\hat{\LogCoe}$ is the maximum likelihood estimator of $\LogCoe$.
        \subsection{Main results}
        Our strategy to obtain the asymptotic distribution of the estimators \eqref{equ:OR} and \eqref{equ:DR} consists of two steps that allow us to apply Prohorov's Theorem \cite[see][Sect. 5, p.57]{Bill79}. 
        First, we consider arbitrary finite-dimensional distributions (i.e. based on a finite discretization on [0,1]) and deduce their joint asymptotic distributions (Lemma \ref{lemma:finite_DR} and \ref{lemma:finite_OR} in the Appendix). 
        These results are the consequence of the theory on M-estimators \cite{NEWEY19942111}. 
        We provide in these lemmas an explicit characterization of the asymptotic covariance function of multivariate OR and DR estimators, which we cannot find in the earlier literature.
        In a second step, given the finite-dimensional results, the functional asymptotic results are established by demonstrating tightness. 
        For this purpose, the re-scaled OR estimator, $\sqrt{n}\Parentheses{\hat{\mu}_{OR} - \mu_y}$, can be rewritten as:
            \begin{equation} \label{OR}
                 \sqrt{n}\Parentheses{\bar{\textbf{x}}_n - \boldsymbol{\mu}_x}^{\top}\boldsymbol{\beta}
                 + \Parentheses{\SumOverN \textbf{x}_i ^{\top}} \Parentheses{\SumOverN Z_i\textbf{x}_i\textbf{x}_i^{\top}}^{-1} \boldsymbol{\zeta} .
            \end{equation}
        where $\boldsymbol{\zeta} = \Parentheses{ n^{-1/2} \SumOverN{ Z_ix_{i,1} \epsilon_i },\dots, n^{-1/2} \SumOverN{ Z_ix_{i,p} \epsilon_i } }^{\top}$.
        Similarly, the re-scaled DR estimator, $\sqrt{n}\Parentheses{\hat{\mu}_{DR} - \mu_y}$, can be represented as
            \begin{equation} \label{DR}
                \sqrt{n}\Parentheses{\hat{\mu}_{OR} - \mu_y} 
                 - \Parentheses{\SumOverN \frac{Z_i}{\tau(\ObsCov^{\top}\hat{\LogCoe})} \textbf{x}_i ^{\top}} \Parentheses{\SumOverN Z_i\textbf{x}_i\textbf{x}_i^{\top}}^{-1} \boldsymbol{\zeta} + n^{-1/2} \SumOverN \frac{Z_i}{\tau(\ObsCov^{\top}\hat{\LogCoe})}\epsilon_i.
            \end{equation}
        Notice that, conditioning on the values of $\Parentheses{Z_i, \textbf{x}_i}_{i = 1}^n$, both (\ref{OR}) and (\ref{DR}) are the sum of independent but non-identical random elements in Hilbert space. 
        Therefore, the tightness of probability measures implied by the conditional re-scaled OR and DR estimators can be established using the following lemma.\\
        \begin{lemma}\label{lemma:LimRandomElements}
            With some probability space $\Parentheses{\Omega, \mathcal{F}, \mathcal{P}}$, let $\Curlybracket{X_n}_{n \geq 1}$ be a sequence of independent random elements in Hilbert space $\Parentheses{\mathcal{H}, \mathcal{B}_{\mathcal{H}}}$ with mean $0$ and $\mathbb{E}\Parentheses{\Norm{X_i}^2} < \infty$ for all $i=1,\dots, n$, and $\xi_n = n^{-1/2}\SumOverN X_i$. The sequence of probability measures implied by $\xi_n$, $\Curlybracket{\mathcal{P}\circ \xi_n^{-1}}_{n\geq1}$, is uniformly tight.  
        \end{lemma}
        To verify tightness without conditioning, we place the re-scaled estimator and $\Parentheses{Z_i, \textbf{x}_i}_{i=1}^n$ in a Cartesian product space. Since we know that the two re-scaled OR and DR estimators are uniformly tight conditioning on the values of $\Parentheses{Z_i, \textbf{x}_i}_{i=1}^n$, then their unconditional tightness is a consequence of the following lemma.
    \begin{lemma}\label{lemma:joint}     
        With some probability space $\Parentheses{\Omega, \mathcal{F}, \mathcal{P}}$, let $\Parentheses{X_n, Y_n}_{n\geq1}$ be a sequence of paired random elements taking values from $\Parentheses{\mathcal{X}, \mathcal{B}_{\mathcal{X}}}$ and $\Parentheses{\mathcal{Y}, \mathcal{B}_{\mathcal{Y}}}$ respectively. 
        Let $\mathbb{P}_n$ be the joint probability measures over $\sigma\Parentheses{\mathcal{B}_{\mathcal{X}} \times \mathcal{B}_{\mathcal{Y}}}$ which is the smallest $\sigma$-algebra making projections to $\mathcal{X}$ and $\mathcal{Y}$ all measurable, i.e. $\mathcal{P}\circ X_n^{-1} \Parentheses{E} = \mathbb{P}_n(E \times \mathcal{Y})$ $\forall E \in \mathcal{B}_{\mathcal{X}}$.
        Assume $\forall y \in \mathcal{Y}$, $\exists$ $\sigma$-algebra $\mathcal{B}_{\mathcal{X}|y}$ and $\mathcal{P} \circ \Parentheses{X_n | y}^{-1}$ such that the joint probability measure can be represented by disintegration, i.e.
        $\mathbb{P}_n \Parentheses{E \times F} = \int_{F} \mathcal{P}\circ\Parentheses{X_n|y}^{-1}(E) d \mathcal{P}\circ Y_n^{-1}$. 
        If $\mathcal{P}\circ \Parentheses{X_n | y}^{-1}$ is uniformly tight, then the probability measure implied by $X_n$ is uniformly tight. Further, if $\mathcal{P}\circ Y_n^{-1}$ is also uniformly tight, then the joint probability measure is uniformly tight. 
    \end{lemma}
        
        It is worth noting that the functional limiting distribution of the OR estimator \eqref{equ:OR} could be obtained by applying functional asymptotic theorems developed in recent years, e.g. \cite{telschow2022simultaneous}. However, the functional DR estimator \eqref{equ:DR} presents more of a challenge, therefore the strategy outlined above is required.  
        To maintain consistency throughout the paper, this strategy is applied to address both estimators \eqref{equ:OR} and \eqref{equ:DR}. This allows us to obtain the following results on the asymptotic distributions for these functional estimators.
            \begin{theorem}\label{thm:tight:OR}
                Assume that the working model \eqref{equ:regmod1} is correctly specified, i.e., $E(\mathcal{Y}_i \mid \mathbf{x}_i) = \mathbf{x}_i^{\top}\boldsymbol{\beta}$. Then the functional OR estimator \eqref{equ:OR} is asymptotically distributed as a Gaussian process:
                \begin{align*}
                    \sqrt{n}(\hat{\mu}_{OR} - \mu_y) \overset{d}{\rightarrow} \mathcal{GP}(0,\boldsymbol{\beta}(s)^{\top} \boldsymbol{\Sigma}_{\mathbf{x}} \boldsymbol{\beta}(t) + \boldsymbol{\mu}_{\mathbf{x}}^{\top} \boldsymbol{\Pi}^{-1} \boldsymbol{\mu}_{\mathbf{x}} \cdot \sigma_\epsilon(s,t)).
                \end{align*}
            \end{theorem}
            \begin{proof}
                By Prohorov's Theorem, this theorem is a consequence of the results in Lemma \ref{lemma:finite_OR}, \ref{lemma:LimRandomElements}, and \ref{lemma:joint}; see Appendix \ref{finite.app} and \ref{tight.app} for details.
            \end{proof}
            \begin{theorem}\label{thm:tight:DR}
                Assume that at least one of the working models \eqref{equ:regmod1} and \eqref{equ:ps} is correctly specified, i.e., $E(\mathcal{Y}_i \mid \mathbf{x}_i) = \mathbf{x}_i^{\top}\boldsymbol{\beta}$ or $\Pr[ Z_i=1 | \mathbf{x}_i ]= \tau(\mathbf{x}_i^{\top} \LogCoe)$.
                Then, the functional DR estimator $\sqrt{n}(\hat{\mu}_{DR} - \mu_y)$ is asymptotically distributed as a Gaussian process with zero mean function. Further, if both working models are correctly specified, the covariance function simplifies and we have:
                \begin{align*}
                    \sqrt{n}(\hat{\mu}_{DR} - \mu_y) \overset{d}{\rightarrow} \mathcal{GP}(0,
                \boldsymbol{\beta}(s)^{\top} \boldsymbol{\Sigma}_{\mathbf{x}} \boldsymbol{\beta}(t) + \mathbb{E} [\tau^{-1}(\mathbf{x}_i^{\top}\LogCoe)]\sigma_\epsilon(s,t)).
                \end{align*}
            \end{theorem}
            \begin{proof}
                By Prohorov's Theorem, this theorem is a consequence of the results in Lemma \ref{lemma:finite_OR}, \ref{lemma:LimRandomElements}, and \ref{lemma:joint}; see Appendix \ref{finite.app} and \ref{tight.app} for details.
            \end{proof}
    
        Note that Theorem \ref{thm:tight:DR} yields a doubly robustness property for the functional estimator, whereby it is asymptotically distributed as a Gaussian process with zero mean function (concentrates at root-$n$ rate) when only one of the nuisance models is correctly specified. 
        Given that we have functional estimators with Gaussian processes as their limiting distribution, we can obtain simultaneous confidence bands as described below.
        \subsection{Feasible estimation and simultaneous confidence bands}\label{sec:feasible}
        The functional estimators \eqref{equ:OR} and \eqref{equ:DR} are not directly feasible since the functional outcomes are infinite dimensional objects. 
        However, considering functional estimators and their derived asymptotic distribution allows us to perform simultaneous inference on feasible counterparts. 
        An example of feasible estimators that we use here are the finite-dimensional estimators based on a discretization of the domain in \eqref{equ:OR2} and \eqref{equ:DR2} in Appendix \ref{finite.app}. 

        Simultaneous confidence bands for the functional parameter of interest, $\mu_y$, can be constructed by applying the results obtained by \cite{Liebl}. 
        This work introduces "fast and fair" simultaneous confidence bands based on the construction of an adaptive, non-constant critical value function, which ensures that false positive rates are balanced across different partitions of the domain. 
        The method enables both global and local interpretations of the bands and is less computationally expensive compared to many simulation-based and resampling-based approaches \cite[e.g.,][]{Pini17, Abram}. 

        To apply the results in \cite{Liebl}, we need to further restrict $\mathcal{Y}_i$ to be once continuously differentiable almost surely in the sequel. 
        Under this assumption and the conditions of Theorem \ref{thm:tight:OR}, a valid 95\% simultaneous confidence band (SCB) for $\mu_{y}$ based on the OR estimator is given by:
            \begin{equation}\label{equ:ffscb}
                \mathrm{SCB_{OR}}(t) := \hat{\mu}_{OR}(t) \pm \mathit{u}^\star_{2.5}(t) \sqrt{C_{OR}(t,t) / n}, 
            \end{equation}
        where $C_{OR}$ is the covariance function of the OR estimator given in Theorem \ref{thm:tight:OR}, and $\mathit{u}^\star_{2.5}(t)$ is a critical value function chosen such that 
            \begin{equation}\label{equ:coverage}
                \Pr\left( \mu_y(t) \in \mathrm{SCB_{OR}}(t) \ \forall \ t \in [0,1]  \right) \geq 0.95.
            \end{equation}
        Simultaneous confidence bands for $\mu_{y}$ based on the DR estimator can be constructed analogously by using $\hat{\mu}_{DR}$ and $C_{DR}$ given in Theorem \ref{thm:tight:DR}. 
        See \cite{Liebl2, Liebl} for further details on the construction and properties of these simultaneous confidence bands.
    \section{Monte Carlo study}
        We here present numerical experiments to study finite sample properties of the estimators in relation to the asymptotic results given above. 
        In particular, we study empirical coverages of the simultaneous confidence bands using the asymptotic covariance functions deduced in Theorems \ref{thm:tight:OR} and \ref{thm:tight:DR}.
        \subsection{Design}
        We generate outcomes $\textbf{y}_i$ following (\ref{equ:regmod2}) (in Appendix \ref{finite.app}) at $T = 50$ equidistant points on $[0,1],$ using the following data generating process: $x_{1i} = 1$; $(x_{2i}, x_{3i}, x_{4i}) \sim \mathcal{MVN} (\mu, \Sigma)$ with $\mu = (-2, 4, 0)$, 
            $$
                \Sigma = 
                \begin{bmatrix}
                     1 & 0.2 & 0.3\\
                    0.2 & 2 & 0.6 \\
                    0.3 & 0.6 & 0.4
                \end{bmatrix};
            $$ 
        $x_{5i} \sim Ber(0.2)$; $x_{6i} \sim Bin(3, 0.6)$.
        The functional coefficients in $\boldsymbol{\beta}$ are set to $40-t, 2\sin(4t), 3 - \cos(5t), 1.5\log(5t + 0.1), 0.5\sin(2t), 2 - 1.5t + 1.3t^2$, respectively, and evaluated at time $t_j, j = 1, \dots, T.$
        The random error term is generated in the first situation as $\boldsymbol{\epsilon}_i \sim \mathcal{MVN} (\mathbf{0}, \Sigma_\epsilon)$. Here, $\Sigma_\epsilon$ is the discretisation of a (stationary) Matérn covariance function so that 
            $$
            \Sigma_{\epsilon (j, k)} = \frac{1}{2^{\kappa - 1} \Gamma(\kappa)} \left( \frac{\mid t_j - t_k \mid}{\phi} \right)^\kappa K_\kappa \left( \frac{\mid t_j - t_k \mid}{\phi} \right),
            $$
        with $\kappa = 1.5$, $\phi = 0.1$, and where $K_\kappa$ is the modified Bessel function of the third kind of order $\kappa$. 
        In a second situation, we generate the error term as $\boldsymbol{\epsilon}_i \sim \mathcal{MVT} (\mathbf{0}, \Delta, \upsilon)$, so that the covariance matrix $\Sigma_\epsilon = \Delta \upsilon/(\upsilon - 2)$, with $\upsilon = 4$.
        The dispersion matrix $\Delta$ is generated as $Q \Lambda Q^{\top}$, where $Q$ is a $T \times T$ arbitrary orthonormal matrix and $\Lambda$ is a $T \times T$ diagonal matrix whose diagonal elements are a sequence of length $T$ ranging from  1 to 3.
        Finally, the missing indicator variable $Z_i$ is generated using (\ref{equ:ps}), with $\LogCoe = (0.3, -0.3, -0.3, -0.3, -0.3, -0.3)^{\top}$, so that $\Pr[ Z_i=1] = 0.71$. 
        This means that on average, around 30\% of the simulated functional outcomes are missing. 
        \begin{figure}[t!]
            \begin{center}
                \includegraphics[width=0.45\linewidth]{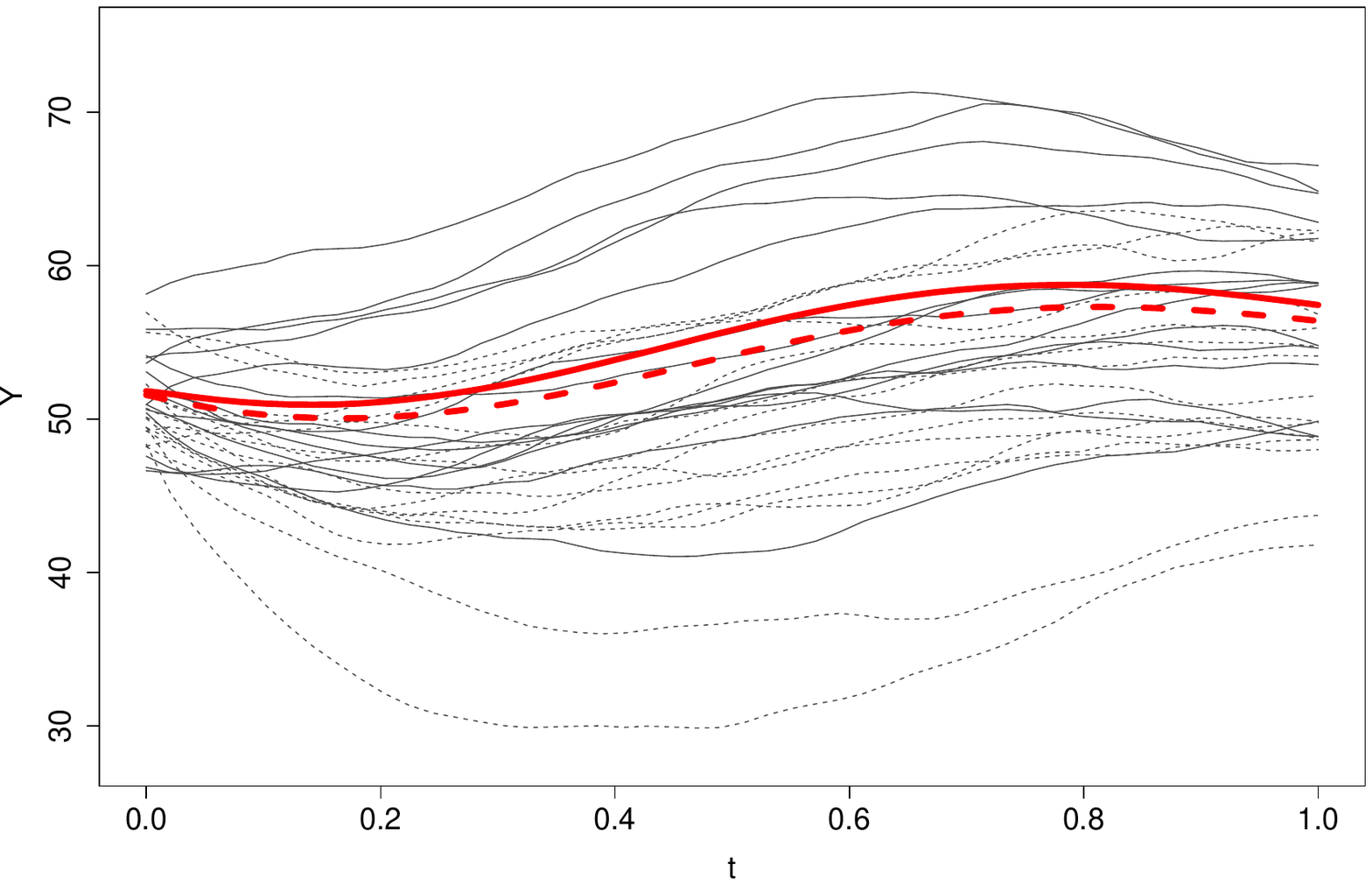} \includegraphics[width=0.45\linewidth]{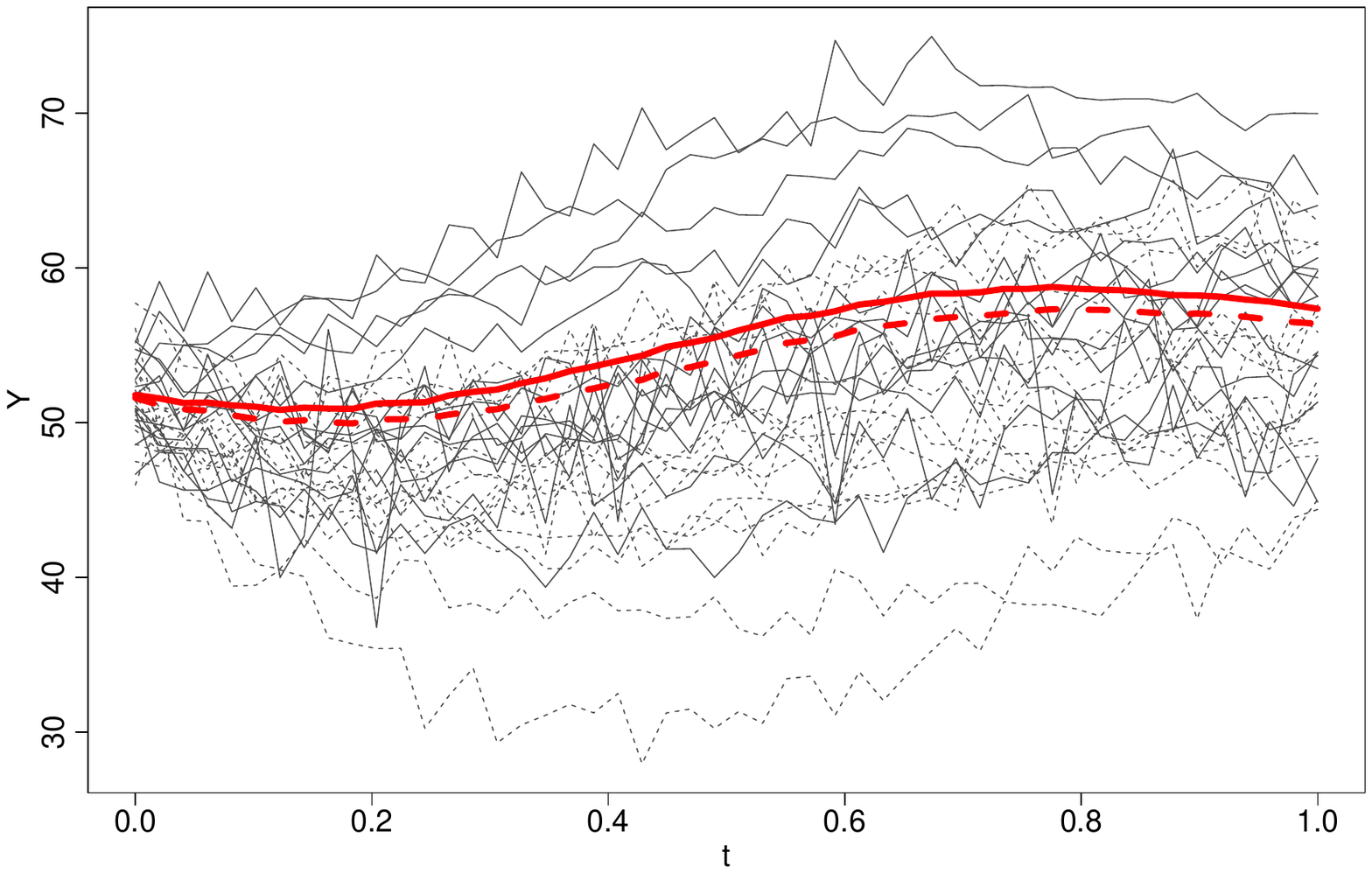} 
            \end{center}
        \caption{Observed (solid grey curves) and "assumed missing" (dotted grey curves) outcomes for 30 generated individuals. Mean function of both observed and missing outcomes (dashed red line) and the naive mean estimate based on observed data (solid red line). Left: Data generated using Gaussian process errors; Right: Data generated using multivariate t-distributed errors.}
        \label{fig:curves_gp}
        \end{figure}
        Figure \ref{fig:curves_gp} shows the outcome functions for 30 generated individuals, using the Gaussian process (left panel) and multivariate t-distribution (right panel) to generate the random errors. 
        The two panels also display the empirical mean function of the outcome using both the observed and the "assumed missing" outcomes (dashed red) and the mean function obtained using only the observed outcomes (solid red). 
        The difference between these curves yields the bias of the naive mean estimator based only on the observed functions. 
        We note the fact that the t-distributed errors yields much less smooth functions, i.e. mimicking a situation where the underlying continuous outcome function is not once continuously differentiable almost surely as assumed in the theory. 

        We use 1000 replicates with sample sizes $n = 250, 500, 1000$ and $3000$. 
        For each replicate, we estimate $\mu_y$ with the outcome regression and doubly robust estimators.
        In order to confirm the doubly robustness property of the latter estimator, we consider four different scenarios. 
        Firstly, when both working models, (\ref{equ:regmod2}) and (\ref{equ:ps}),  are correctly specified; secondly, when only the outcome regression model, (\ref{equ:regmod2}), is misspecified. 
        Thirdly, when only the model for the missingness mechanism, (\ref{equ:ps}), is misspecified, and lastly when both working models are misspecified. 
        In each case, the misspecification entails that the model(s) omit(s) two of the covariates, $x_3$ and $x_5$. 
        The performance of the estimators and their inference is evaluated with respect to bias, variance, mean squared error (MSE), and the empirical coverage of 95\% confidence bands, both pointwise (PCB) and simultaneous bands (SCB). 
        The SCB is constructed using sample variants of the covariance functions given in Theorem \ref{thm:tight:OR} and \ref{thm:tight:DR}. 
        This is done using the \texttt{R}-package \texttt{ffscb} \cite{ffscb}, implementing the computation of the functional critical value $\mathit{u}^\star_{.025}(t)$. 
        The PCB instead uses a constant critical value $t_{.025}$, the 97.5th percentile from the t-distribution with $n-1$ degrees of freedom. 
        In addition, we compare the performance of the estimators to the naive estimate based on complete case analysis, i.e. calculating the sample mean and its covariance based only the observed functional outcomes. 

        \begin{table}[t!]\label{table:simGP}
            \caption{Empirical coverages of the simultaneous confidence bands (SCB) and pointwise confidence bands (PCB) for the OR- and DR-estimator in the presence of different types of model misspecification compared to a complete case analysis. Based on 1000 simulation replicates with Gaussian error terms. The nominal coverage level is 95\%.}
            \begin{center}
                \resizebox{0.95\textwidth}{!}{
                \begin{tabular}{l l c c  c c c c c c}
                      \hline
                      &  & \multicolumn{2}{l}{\textbf{OR-estimator}} & &\multicolumn{4}{l}{\textbf{DR-estimator}} & \textbf{Complete Cases}\\
                      \multicolumn{2}{c}{\textbf{Misspecification:}} & None & Outcome && None & Outcome & Missingn. & Both & - \\
                      \hline
                     \multirow{2}{4em}{n = 250} & SCB & 98.1 & 92.1   && 98.1 & 98.5 & 98.2 & 91.2 & 12.3\\
                     & PCB & 83.2 & 69.2  && 82.9 & 86.0 & 82.6 & 70.0 & 3.8\\
                      \hline 
                      \multirow{2}{4em}{n = 500} & SCB & 97.3 & 86.7  && 97.2 & 98.0 & 97.2 & 86.2 & 0.8\\
                     & PCB & 81.9 & 58.6 && 81.9 & 86.0 & 81.5 & 58.1 & 0.2\\
                      \hline
                      \multirow{2}{4em}{n = 1000} & SCB & 97.5 & 74.0  && 97.5 & 98.0 & 97.4 & 73.0 & 0\\
                     & PCB & 81.0 & 39.0 && 81.2 & 84.9 & 80.9 & 39.0 & 0\\
                      \hline
                     \multirow{2}{4em}{n = 3000} & SCB & 97.0 & 26.0  && 97.2 & 97.9 & 96.9 & 25.4 & 0\\
                     & PCB & 81.0 & 6.1 && 81.4 & 84.9 & 81.1 & 6.2 & 0\\
                      \hline
                \end{tabular}
                }
            \end{center}
        \end{table}
        \begin{table}[h!]\label{table:simMVT}
            \caption{Empirical coverages of the simultaneous confidence bands (SCB) and pointwise confidence bands (PCB) for the OR- and DR-estimator in the presence of different types of model misspecification compared to a complete case analysis. Based on 1000 simulation replicates with multivariate t-distributed error terms. The nominal coverage level is 95\%.}
            \begin{center}
            \resizebox{0.95\textwidth}{!}{
                \begin{tabular}{l l c c  c c c c c c}
                      \hline
                      &  & \multicolumn{2}{l}{\textbf{OR-estimator}} & &\multicolumn{4}{l}{\textbf{DR-estimator}} & \textbf{Complete Cases}\\
                      \multicolumn{2}{c}{\textbf{Misspecification:}} & None & Outcome && None & Outcome & Missingn. & Both & - \\
                      \hline
                     \multirow{2}{4em}{n = 250} & SCB & 97.9 & 94.0 && 97.7 & 98.6 & 97.5 & 94.1 & 7.2\\
                     & PCB & 62.8 & 51.6 && 62.7 & 69.5 & 62.9 & 50.9 & 0.8\\
                      \hline 
                      \multirow{2}{4em}{n = 500} & SCB & 97.6 & 89.9 && 97.6 & 98.6 & 97.6 & 89.6 & 1.3\\
                     & PCB & 65.5 & 39.9 && 64.8 & 70.2 & 65.1 & 40.1 & 0\\
                      \hline
                      \multirow{2}{4em}{n = 1000} & SCB & 98.6 & 79.5 && 98.6 & 99.1 & 98.5 & 79.9 & 0\\
                     & PCB & 63.1 & 24.8 && 62.5 & 68.1 & 61.8 & 25.2 & 0 \\
                      \hline
                     \multirow{2}{4em}{n = 3000} & SCB & 98.1 & 31.0 && 98.1 & 98.8 & 98.1 & 30.8 & 0\\
                     & PCB & 63.2 & 2.6 && 62.5 & 70.0 & 62.6 & 2.4 & 0 \\
                      \hline
                    \end{tabular}            
            }

            \end{center}
        \end{table}
        \subsection{Results}
        Tables 1 and 2 display the empirical coverages of the pointwise and simultaneous confidence bands using a 95\% nominal level, for the Gaussian and t-distributed errors, respectively. 
        As expected, the 95\% pointwise confidence bands perform much worse overall, since they do not take the multitude of comparisons into account. 
        Their coverage deteriorates drastically with sample size, as would also be expected since sampling variation decreases.

        When the model is correctly specified, the simultaneous confidence bands based on the OR-estimator achieve an empirical coverage bounded below by the nominal level of 95\% (as expected from the theory for the Gaussian noise case; see (\ref{equ:coverage})). 
        The same holds for the simultaneous confidence bands based on the DR-estimator when at least one of the two nuisance models is correctly specified. Note that our assumption of an outcome once continuously differentiable almost surely is violated when simulating t-distributed errors, and yet the empirical coverages of the simultaneous confidence bands behave well. 
        When the estimators are asymptotically biased (the outcome model is misspecified for the OR-estimator, or both models are misspecified for the DR-estimator), the nominal coverage levels are not reached, as expected. The confidence bands for the naive estimate based on complete case analysis predictably perform badly, due to large bias when around 30\% of the curves are missing. 
        
        Figure \ref{fig:bandwidth} in Appendix \ref{figure.app} illustrates the width of the pointwise compared to the simultaneous confidence bands. Figures \ref{fig:bias_gp}-\ref{fig:mse_t} in the Appendix \ref{figure.app} display biases, variances, and MSE for the estimators in the different scenarios. 
        There we can see that, as expected, the bias of the DR estimator is negligible when one or both of the two working models are correctly specified, similarly for the OR estimator when the outcome model is correct. 
        When both models are misspecified, the DR estimator displays some bias, which varies between 0.17 and 0.38. 
        A similar bias is observed for the OR estimator when the outcome model is incorrect. 
        For comparison, the bias of the naive estimate varies between 0.2 and 1.39. 
        These results confirm the expected doubly robustness properties of the DR estimator.
    \section{Application: Answering a counterfactual question}\label{sec:app}
        To illustrate the introduced methods we revisit a cohort dataset including all people born in Sweden year 1954. 
        This data was used earlier in \cite{ecker2023causal} to contrast the lifetime income trajectories between those living at 20 years of age in large labour market regions compared with those living in small labour market regions. 
        Local labour markets are units defined by patterns of employment and commuting at the municipality level \cite{LA2,LA}. 

        Counterfactual questions are often of interest in the social sciences and here we focus on the following one for illustration purposes: what would be the average life income trajectory for this cohort if all had been living in a large labour market region at age 20 (i.e., in 1974)? 
        For this purpose, let $\cal Y$ denote the lifetime income trajectory of an individual between age 21 and age 63 (we can follow the cohort until year 2017 in the data). 
        We observe $\cal Y$ for those actually living in the labour market of one of the three biggest Swedish cities (Stockholm, Gothenburg, or Malmö, $Z=1$). 
        For those living in smaller labour markets ($Z=0$) we do not observe the counterfactual life income trajectory that they would have had if they had lived in the larger labour market. The functional parameter corresponding to the counterfactual question of interest is $\mu_{y}=E(\cal Y)$. 

        More precisely, the outcome of interest $\cal Y$ measured annually is the logarithm of the accumulated total earned income in SEK, defined as all taxable income (except capital). We adjust incomes for inflation to match the monetary value in 2017, discounting with a factor of 0.03 \cite{disk}. 
        
        We assume the missingness mechanism is ignorable ($\Pr[Z=1\mid X,{\cal Y}]=\Pr[Z=1\mid X]$), where $X$ contains the following covariates: \textit{sex};
        \textit{year of first income}, the year of first income if this occurs before exposure (i.e., before 1974); \textit{previous income}, the logarithmized cumulative income accumulated before exposure; \textit{children before age 20}, a binary indicator of having at least one child before exposure; and \textit{secondary education}, an indicator of completing (at least) upper secondary education. The latter is based on the subjects' highest achieved education level in 1990, the first year for which the variable is available annually. 
        In addition, we control for some family measures: \textit{number of siblings} (in 1973); \textit{parents' income} during adolescence (the logarithmized total for both parents for the years 1968-1971); and an indicator indicating if both parents are born outside of Sweden (\textit{foreign-born parents}). 
        We also control for the highest education level (lower secondary, upper secondary, or tertiary) achieved by either mother, father, or both, in 1970, yielding the indicator variables \textit{parents' secondary education} and \textit{parents' tertiary education}.
        Note that we excluded a total of around 4000 individuals due to missingness on one or more of the covariates, primarily related to parental characteristics. 
        The resulting data consists of 54485 individuals, of which 57.4\% have observed outcomes.

        \begin{figure}[h!]
            \begin{center}
            \includegraphics[width=0.9\linewidth]{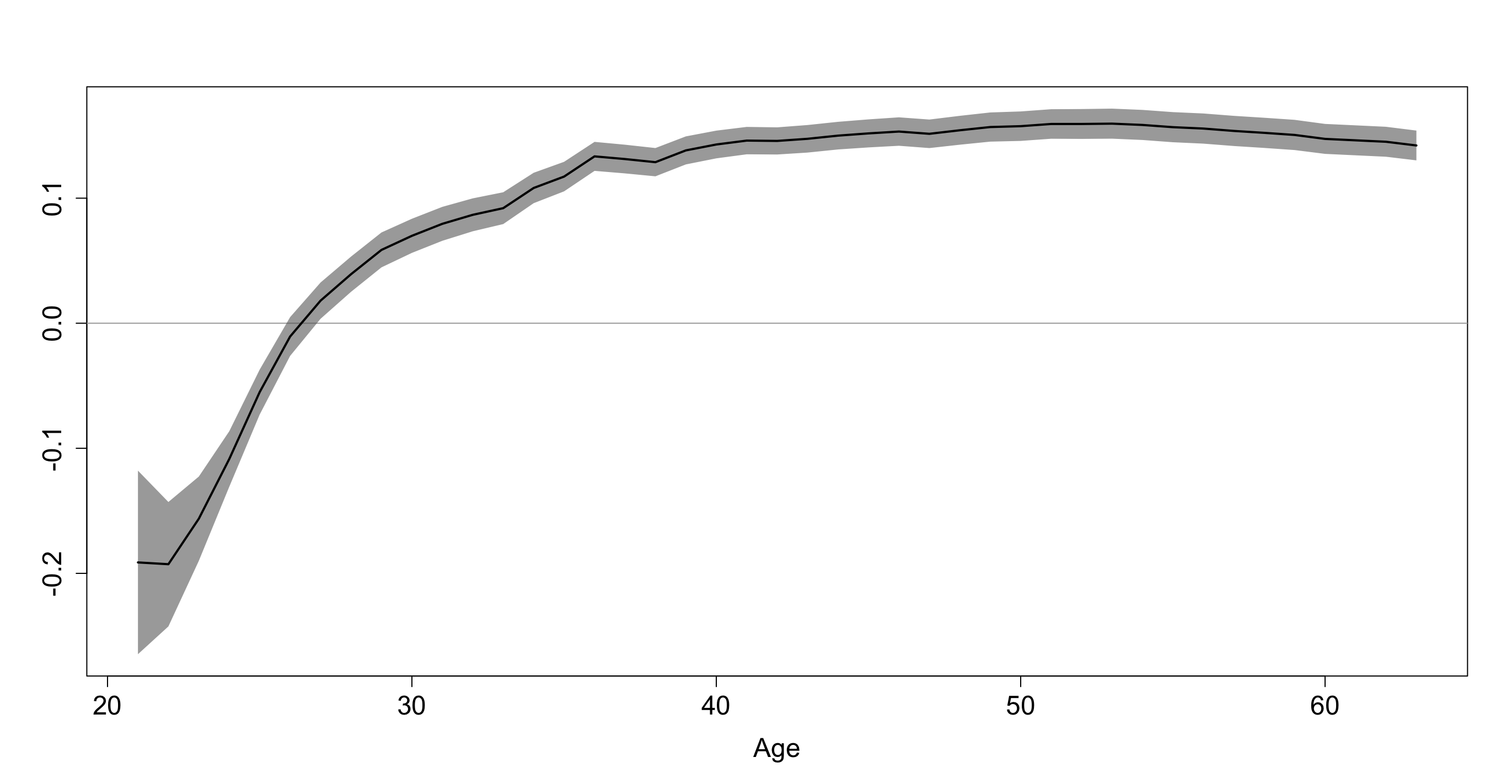} 
            \end{center}
            \caption{Solid line: $\hat \mu_{DR}-\widehat E({\cal Y} \mid Z=1)$; and shaded grey: simultaneous 95\% confidence bands.}
            \label{app.fig}
        \end{figure}

        We estimate $\mu_y$ with the doubly robust estimator. In Figure \ref{app.fig} we display the comparison $\hat \mu_{DR}-\widehat E({\cal Y} \mid Z=1)$, i.e. comparing $\hat \mu_{DR}$ to the life income trajectory of those that actually started their labour market career in one of the large city. 
        We see with this comparison that the naive estimation overestimates the income trajectory in the early years, but already before 30 years of age, it starts underestimating accumulated incomes. 
        Confidence bands are obtained using the results of Theorem \ref{thm:tight:DR} and (\ref{equ:coverage}) applied to the doubly robust estimator.

        \section{Discussion}
        In this paper, we have introduced and studied two semi-parametric estimators for the mean of functional outcomes in the presence of missing (at random) outcome data. 
        We have shown that both of these estimators (an outcome regression- and a doubly robust estimator) are asymptotically Gaussian processes, and presented expressions for their covariance functions. This enables us to quantify the uncertainty of these estimators by constructing simultaneous confidence bands with asymptotically guaranteed coverage. 
        These asymptotics translate into satisfactory finite sample performances in the Monte Carlo study conducted. 
        Further, the estimators studied here can be used in the context of causal inference from observational data, using the potential outcomes framework \cite{missingcausal}, as illustrated by the application in Section \ref{sec:app}.    

        A central assumption we have made is that the functional outcomes are missing at random. While this assumption is commonly required in the literature, it is nevertheless restrictive and cannot be tested without further assumptions \cite{molenberghs2015handbook}. 
        Future work could thus consider developing uncertainty intervals that take into account departures from the missing at random assumption \cite{vansteelandt2006ignorance, genback2015uncertainty}. 
        Other possible extensions include considering the targeted learning estimator \citep{tmle}, which is closely related to the doubly robust estimator, and more generally, developing semi-parametric theory in the context of functional outcomes so that a larger class of machine learning methods can be used to fit the nuisance models.


\section*{Acknowledgments}
We are thankful for comments by David Källberg and two anonymous referees that have improved the paper.
This work was funded by a grant from the Marianne and Marcus Wallenberg Foundation.

\bibliographystyle{plain}
\bibliography{ref}

\newpage

\appendix
\noindent{\LARGE\textbf{Appendix}}
\section{Asymptotic results for finite-dimensional distributions}\label{finite.app} 
In addition to the estimators in \eqref{equ:OR} and \eqref{equ:DR}, we consider finite-dimensional analogs, where we evaluate the outcome function at $T$ points on $[0,1]$. 
    These finite-dimensional analogs are an important step in the proofs of our main results and are also needed for the feasible estimation presented in Section \ref{sec:feasible}.    
    It is equivalent to consider a random sample, $(\ObsRes, \ObsCov, Z_i)_{i=1}^n$, from the joint distribution of $(\textbf{y}, \textbf{x}, Z)$, where the outcome variable, $\textbf{y}$, is a $T \times 1$ random vector.
    Corresponding to working model \eqref{equ:regmod1}, we have the multivariate multiple regression,
        \begin{equation}\label{equ:regmod2}
            \ObsRes = \textbf{B}^{\top} \ObsCov + \boldsymbol{\epsilon}_i
        \end{equation}
    where $\textbf{B} = (\RegCoe_1, \dots, \RegCoe_T)$, $\RegCoe_j$ is a $p \times 1$ vector,  and $\boldsymbol{\epsilon}_i$ is a random vector with zero mean vectors and $\boldsymbol{\Sigma}_{\epsilon}$ as the covariance matrix.    
    The OR estimator of $\boldsymbol{\mu}_{y} = \mathbb{E}(\textbf{y})$ based on $(\ObsRes, \ObsCov, Z_i)_{i=1}^n$ and \eqref{equ:regmod2} is
        \begin{align}\label{equ:OR2}
            \hat{\boldsymbol{\mu}}_{OR} = n^{-1}\SumOverN \hat{\textbf{B}}^{\top}\ObsCov,
        \end{align}
    where $\hat{\textbf{B}}$ is the ordinary least squares estimator.
    Similarly, based on the two working models \eqref{equ:ps} and \eqref{equ:regmod2}, the corresponding doubly robust estimator is
        \begin{equation}\label{equ:DR2}
            \hat{\boldsymbol{\mu}}_{DR} = n^{-1} \SumOverN \left( \hat{\textbf{B}}^{\top}\ObsCov + \frac{Z_i}{\tau(\ObsCov^{\top}\hat{\LogCoe})}(\ObsRes -\hat{\textbf{B}}^{\top}\ObsCov)  \right),
        \end{equation}
    where $\hat{\LogCoe}$ again is the maximum likelihood estimator. The above estimators can be viewed as M-estimators and we assume throughout the usual regularity conditions for such estimators to be well-behaved asymptotically; see \cite{NEWEY19942111} and \citet[][Suppl Material]{CANTONI2020108}.
    \begin{lemma}\label{lemma:finite_DR}
       Assume that at least one of the two working models \eqref{equ:regmod1} and \eqref{equ:ps} is correctly specified, i.e., $E(\ObsRes \mid \ObsCov)= \textbf{B}^{\top} \ObsCov$ or $\Pr[ Z_i=1 | \ObsCov ]= \tau(\ObsCov^{\top} \LogCoe)$. Then,  $\sqrt{n}(\hat{\boldsymbol{\mu}}_{DR} - \boldsymbol{\mu}_y)$ has an asymptotic multivariate normal distribution with mean zero. In the case that both working models are correctly specified, the covariance matrix simplifies and we have:
       \begin{align*}
           \sqrt{n}(\hat{\boldsymbol{\mu}}_{DR} - \boldsymbol{\mu}_y) \overset{d}{\rightarrow} \mathcal{N}\left(\textbf{0}, \textbf{B}^{\top} \boldsymbol{\Sigma}_{x} \textbf{B} + \mathbb{E} [\tau^{-1}(\ObsCov^{\top}\LogCoe)] \boldsymbol{\Sigma}_{\epsilon}\right).
       \end{align*} 
    \end{lemma}
    \begin{proof}       
        Based on the two models, the doubly robust estimator of the expected value of the outcome variables, $\boldsymbol{\mu}_{y} = \mathbb{E}(\textbf{y})$, is defined as
            \begin{equation*}
                \hat{\boldsymbol{\mu}}_{DR} = n^{-1} \SumOverN \left( \hat{\textbf{B}}^{\top}\ObsCov + \frac{Z_i}{\tau(\ObsCov^{\top}\hat{\LogCoe})}(\ObsRes -\hat{\textbf{B}}^{\top}\ObsCov)  \right),
            \end{equation*}
        where $\hat{\textbf{B}}$ and $\hat{\LogCoe}$ are maximum likelihood estimators.
        We apply the M-estimation approach to derive the asymptotic distribution of the doubly robust estimator.
        To do so, we first define the vector equations $G_n(\boldsymbol{\theta}) = n^{-1}\SumOverN{\psi_i(D_i, \boldsymbol{\theta})}$, where $D_i$ denotes the data of $i$th observation, and $\boldsymbol{\theta} = (\boldsymbol{\mu}_y, \vect(\textbf{B}), \LogCoe)^{\top}$ denotes the vector of all parameters. 
        To simplify the notations, we ignore $D_i$ in the equation $\psi_i$.
        The estimator $\hat{\boldsymbol{\theta}}$, which is the solution of the vector equations $G_n(\boldsymbol{\theta}) = 0$, has the following limiting distribution: 
            \begin{equation*}
                \sqrt{n}(\hat{\boldsymbol{\theta}} - \boldsymbol{\theta}_0) \to \mathcal{MN}(\textbf{0}, \boldsymbol{\Sigma}_{\theta}),
            \end{equation*}
        where $\boldsymbol{\Sigma}_{\theta} = \mathbb{E} \left( \frac{\partial \psi_i(\boldsymbol{\theta})}{ \partial \boldsymbol{\theta}} \biggr\rvert_{\boldsymbol{\theta}_0} \right)^{-1} 
                    \mathbb{E} [ \psi_i(\boldsymbol{\theta}_0)\psi_i(\boldsymbol{\theta}_0)^{\top} ]
                    \left\{ \mathbb{E} \left( \frac{\partial \psi_i(\boldsymbol{\theta})}{ \partial \boldsymbol{\theta}} \biggr\rvert_{\boldsymbol{\theta}_0} \right)^{-1} \right\}^{\top}$. \\                
        In order to define the estimating equations in terms of $\vect(\textbf{B})$, we need to introduce an equivalent representation of the model (\ref{equ:regmod2}), which is:
            \begin{equation*}
                \ObsRes = (\IdenMat_{T \times T} \otimes \ObsCov^{\top})\vect(\textbf{B}) + \boldsymbol{\epsilon}_i.
            \end{equation*}
        Based on this representation we can define the following equations, which together make up $\psi_i(\boldsymbol{\theta})$:
            \begin{equation}\label{equ:psi1}
                \psi_{1,i}(\boldsymbol{\theta}) = (\IdenMat_{T \times T} \otimes \ObsCov^{\top})\vect(\textbf{B}) + \frac{Z_i}{\tau(\ObsCov^{\top}\LogCoe)}(\textbf{y}_i - (\IdenMat_{T \times T} \otimes \ObsCov^{\top})\vect(\textbf{B})) - \boldsymbol{\mu}_{y}
            \end{equation}    
            \begin{equation}\label{equ:psi2}
                \psi_{2,i}(\boldsymbol{\theta}) = (\ObsRes - (\IdenMat_{T \times T} \otimes \ObsCov^{\top})\vect(\textbf{B})) \otimes Z_i \ObsCov
            \end{equation}    
            \begin{equation}\label{equ:psi3}
                \psi_{3,i}(\boldsymbol{\theta}) = \frac{Z_i - \tau(\ObsCov^{\top} \LogCoe)}{\tau(\ObsCov^{\top} \LogCoe)(1-\tau(\ObsCov^{\top} \LogCoe))} \frac{\partial \tau(\ObsCov^{\top} \LogCoe)}{ \partial \LogCoe}.
            \end{equation}
        Note: (\ref{equ:psi1}), (\ref{equ:psi2}) and (\ref{equ:psi3}) all are vector valued functions and the dimensions are $T \times 1$, $Tp \times 1$, and $p \times 1$, respectively.
        Calculate the Jacobian of $\psi_i(\boldsymbol{\theta})$ and evaluate it at the true value $\boldsymbol{\theta}_0$.
        The first block of rows of the Jacobian is
            \begin{equation*}
                \frac{\partial \psi_{1,i}(\boldsymbol{\theta})}{ \partial \boldsymbol{\theta}^{\top} } \biggr\rvert_{\boldsymbol{\theta}_0}  
                = \left[ - \IdenMat_{T \times T}, \frac{\tau(\ObsCov^{\top} \LogCoe_0)- Z_i}{\tau(\ObsCov^{\top} \LogCoe_0 )} (\IdenMat_{T \times T} \otimes \ObsCov^{\top}), -Z_i e^{-\ObsCov^{\top} \LogCoe_0} \boldsymbol{\epsilon}_i \ObsCov^{\top} \right].
            \end{equation*}        
        The second block of rows is
            \begin{equation*}
                \frac{\partial \psi_{2,i}(\boldsymbol{\theta})}{ \partial \boldsymbol{\theta}^{\top} } \biggr\rvert_{\boldsymbol{\theta}_0} 
                = \left[ \textbf{0}_{Tp \times T},  Z_i \IdenMat_{T \times T} \otimes \textbf{x}_i\textbf{x}_i^{\top} , \textbf{0}_{p \times p} \right].
            \end{equation*}
        The last block of rows is
            \begin{equation*}
                \frac{\partial \psi_{3,i}(\boldsymbol{\theta})}{ \partial \boldsymbol{\theta}^{\top} } \biggr\rvert_{\boldsymbol{\theta}_0} 
                = \left[ \textbf{0}_{p \times T}, \textbf{0}_{p \times Tp}, \tau(\ObsCov^{\top} \LogCoe_0)(1-\tau(\ObsCov^{\top} \LogCoe_0)) \ObsCov\ObsCov^{\top}  \right].
            \end{equation*}
        
        If both working models are correctly specified, the expected value of the Jacobian is 
            \begin{equation*}
                \mathbb{E} \left( \frac{\partial \psi_i(\boldsymbol{\theta})}{ \partial \boldsymbol{\theta}} \biggr\rvert_{\boldsymbol{\theta}_0} \right) 
                    = \begin{bmatrix}
                    -\IdenMat_{T \times T} & \textbf{0}_{T \times Tp} & \textbf{0}_{T \times p}\\ 
                    \textbf{0}_{Tp \times T} & \IdenMat_{T \times T} \otimes \boldsymbol{\Pi}_{p \times p} & \textbf{0}_{p \times p}\\ 
                    \textbf{0}_{p \times T} & \textbf{0}_{p \times Tp} & \boldsymbol{\Phi}_{p \times p}
                    \end{bmatrix},
            \end{equation*}
        where $\boldsymbol{\Pi} = \mathbb{E}\left[ Z_i\textbf{x}_i\textbf{x}_i^{\top} \right]$ and $\boldsymbol{\Phi} = \mathbb{E}\left[\tau(\ObsCov^{\top} \LogCoe_0)(1-\tau(\ObsCov^{\top} \LogCoe_0)) \ObsCov\ObsCov^{\top}\right]$.
        To see why this is the case, note that if the missingness model \eqref{equ:ps} is correctly specified, the second block in the first row of the Jacobian simplifies as follows: 
            \begin{align*}
                & \mathbb{E}\left(\frac{\tau(\ObsCov^{\top} \LogCoe_0)- Z_i}{\tau(\ObsCov^{\top} \LogCoe_0)} (\IdenMat_{T \times T} \otimes \ObsCov^{\top})\right) \\
                &= \mathbb{E}_{\textbf{x}} \left( \mathbb{E} \left( \frac{\tau(\ObsCov^{\top} \LogCoe_0)- Z_i}{\tau(\ObsCov^{\top} \LogCoe_0)} (\IdenMat_{T \times T} \otimes \ObsCov^{\top}) \biggr\rvert \ObsCov \right) \right) = 0
            \end{align*}
        If the outcome model \eqref{equ:regmod2} is correctly specified, the third block in the first row of the Jacobian simplifies as follows:
            \begin{equation*}
                \mathbb{E} \left( -Z_i e^{-\ObsCov^{\top} \LogCoe_0} \boldsymbol{\epsilon}_i \ObsCov^{\top} \right) = \mathbb{E}\boldsymbol{\epsilon}_i \mathbb{E}\left( -Z_i e^{-\ObsCov^{\top} \LogCoe_0} \ObsCov^{\top} \right) = \textbf{0}. \\
            \end{equation*}    
        Next, calculate $\mathbb{E} [ \psi_i(\boldsymbol{\theta}_0)\psi_i(\boldsymbol{\theta}_0)^{\top} ] $.
        Still assuming both working models are correctly specified, we have: 
            \begin{align*}
                &\mathbb{E} [ \psi_{1,i}(\boldsymbol{\theta}_0) \psi_{1,i}(\boldsymbol{\theta}_0)^{\top} ] \\ 
                &= \Cov\left[ (\IdenMat_{T \times T} \otimes \ObsCov^{\top})\vect(\textbf{B}_0) \right] + \mathbb{E} \left[ \frac{Z_i}{\tau^2(\ObsCov^{\top}\LogCoe_0)} \boldsymbol{\epsilon}_i\boldsymbol{\epsilon}_i^{\top} \right] \\
                &= \textbf{B}^{\top}_0 \boldsymbol{\Sigma}_{x} \textbf{B}_0 + \mathbb{E} [\tau^{-1}(\ObsCov^{\top}\LogCoe_0)] \boldsymbol{\Sigma}_{\epsilon},
            \end{align*}
        where 
            \begin{align*}
                \mathbb{E}\left[ \frac{1}{\tau(\ObsCov^{\top}\LogCoe_0)} \right] = \frac{1}{\pi}\left( 1 - \frac{\mathbb{V}(\tau(\ObsCov^{\top}\LogCoe_0))}{ [\mathbb{E}\tau(\ObsCov^{\top}\LogCoe_0)]^2 } \right) = \frac{1}{\pi} \left(1 - \frac{\mathbb{V}(\tau(\ObsCov^{\top}\LogCoe_0))}{\pi^2} \right),
            \end{align*}
            with $\pi=E(Z_i)$,
            and
            \begin{align*}
                \mathbb{V}(\tau(\ObsCov^{\top}\LogCoe_0)) &= \mathbb{E}_{\textbf{x}}[ \tau(\ObsCov^{\top}\LogCoe_0)^2 ] - \pi^2.
            \end{align*}
        Next,        
            \begin{align*}
                &
                \mathbb{E} [ \psi_{1,i}(\boldsymbol{\theta}_0) \psi_{2,i}(\boldsymbol{\theta}_0)^{\top} ] \\
                &= \mathbb{E} \left[ \left( (\IdenMat_{T \times T} \otimes \ObsCov^{\top})\vect(\textbf{B}_0) - \boldsymbol{\mu}_y^0 + \frac{Z_i}{\tau(\ObsCov^{\top}\LogCoe_0)}\boldsymbol{\epsilon}_i \right) \left( \boldsymbol{\epsilon}_i \otimes Z_i \ObsCov \right)^{\top} \right] \\
                &= \mathbb{E} \left[ \frac{Z_i}{\tau(\ObsCov^{\top}\LogCoe_0)} \boldsymbol{\epsilon}_i  \left( \boldsymbol{\epsilon}_i \otimes Z_i \ObsCov \right)^{\top} \right] 
                = \mathbb{E} \left[ \frac{Z_i}{\tau(\ObsCov^{\top}\LogCoe_0)}  ( \boldsymbol{\epsilon}_i \boldsymbol{\epsilon}_i^{\top} \otimes \ObsCov^{\top} ) \right] \\
                &= \boldsymbol{\Sigma}_{\epsilon} \otimes\boldsymbol{\mu}_x^{\top}
            \end{align*}
        and 
            \begin{align*}
                &\mathbb{E} [ \psi_{1,i}(\boldsymbol{\theta}_0) \psi_{3,i}(\boldsymbol{\theta}_0)^{\top} ] \\
                &= \mathbb{E} \left[ \left( (\IdenMat_{T \times T} \otimes \ObsCov^{\top})\vect(\textbf{B}_0) - \boldsymbol{\mu}_y^0 + \frac{Z_i}{\tau(\ObsCov^{\top}\LogCoe)}\boldsymbol{\epsilon}_i \right) (Z_i - \tau(\ObsCov^{\top}\LogCoe)) \ObsCov^{\top} \right]
                = \textbf{0}. \\
            \end{align*}
        Next,
            \begin{align*}
                &
                \mathbb{E} [ \psi_{2,i}(\boldsymbol{\theta}_0) \psi_{2,i}(\boldsymbol{\theta}_0)^{\top} ] \\
                &= \mathbb{E} \left[  (\boldsymbol{\epsilon}_i\otimes Z_i\ObsCov) (\boldsymbol{\epsilon}_i\otimes Z_i\ObsCov)^{\top} \right]
                = \mathbb{E}  \boldsymbol{\epsilon}_i\boldsymbol{\epsilon}_i^{\top} \otimes \left[ Z_i\ObsCov\ObsCov^{\top}  \right] \\
                &=\mathbb{E}\left[\boldsymbol{\epsilon}_i\boldsymbol{\epsilon}_i^{\top} \right] \otimes \mathbb{E} \left[Z_i \ObsCov\ObsCov^{\top}\right] 
                =   \boldsymbol{\Sigma}_{\epsilon} \otimes \boldsymbol{\Pi} 
            \end{align*}
        and 
            \begin{align*}
                \mathbb{E} [ \psi_{2,i}(\boldsymbol{\theta}_0) \psi_{3,i}(\boldsymbol{\theta}_0)^{\top} ] 
                =  \mathbb{E} \left[  ( \boldsymbol{\epsilon}_i \otimes Z_i\ObsCov ) (Z_i - \tau(\ObsCov^{\top} \LogCoe))\ObsCov^{\top} \right] = \textbf{0}.
            \end{align*}
        Since $\hat{\LogCoe}$ is a maximum likelihood estimator, the blocks in $\mathbb{E} [ \psi_{3,i}(\boldsymbol{\theta}_0)\psi_{3,i}(\boldsymbol{\theta}_0)^{\top} ] $ are identical to the expected value of the corresponding blocks in the Jacobian.
        According to the M-estimation approach, we have $\sqrt{n}( \hat{\boldsymbol{\theta}} -\boldsymbol{\theta} ) \overset{d}{\rightarrow} \mathcal{N}(\textbf{0}, \boldsymbol{\Sigma_\theta})$,
        where, by the sandwich estimator formula, the asymptotic variance of $\hat{\boldsymbol{\theta}}$ is
            \begin{align*}
                &  \mathbb{E} \left( \frac{\partial \psi_i(\boldsymbol{\theta})}{ \partial \boldsymbol{\theta}} \biggr\rvert_{\boldsymbol{\theta}_0} \right)^{-1} 
                    \mathbb{E} [ \psi_i(\boldsymbol{\theta}_0)\psi_i(\boldsymbol{\theta}_0)^{\top} ]
                    \left\{ \mathbb{E} \left( \frac{\partial \psi_i(\boldsymbol{\theta})}{ \partial \boldsymbol{\theta}} \biggr\rvert_{\boldsymbol{\theta}_0} \right)^{-1} \right\}^{\top} \\
                &=  \begin{bmatrix}
                        \textbf{B}^{\top}_0 \boldsymbol{\Sigma}_{x} \textbf{B}_0 + \mathbb{E} [\tau^{-1}(\ObsCov^{\top}\LogCoe_0)] \boldsymbol{\Sigma}_{\epsilon} &  -\boldsymbol{\Sigma}_{\epsilon} \otimes \boldsymbol{\mu}_x^{\top} \boldsymbol{\Pi}^{-1}  & \textbf{0} \\ 
                        \boldsymbol{\Sigma}_{\epsilon} \otimes \boldsymbol{\Pi}^{-1}\boldsymbol{\mu}_x & \boldsymbol{\Sigma}_{\epsilon} \otimes \boldsymbol{\Pi}^{-1} & \textbf{0} \\ 
                        \textbf{0} & \textbf{0} & \boldsymbol{\Phi}^{-1} 
                    \end{bmatrix}.
            \end{align*}
        Then this lemma is complete by the delta method.
    \end{proof}
    \begin{lemma}\label{lemma:finite_OR}
       Assume the working model \eqref{equ:regmod1} is correctly specified, i.e., $E(\ObsRes \mid \ObsCov)= \textbf{B}^{\top} \ObsCov$. Then, the OR estimator \eqref{equ:OR2} has the following asymptotic multivariate normal distribution:
       \begin{align*}
           \sqrt{n}(\hat{\boldsymbol{\mu}}_{OR} - \boldsymbol{\mu}_y) \overset{d}{\rightarrow} \mathcal{N}\left(\textbf{0}, \textbf{B}^{\top} \boldsymbol{\Sigma}_{x} \textbf{B} + \boldsymbol{\mu}_x^{\top} \boldsymbol{\Pi}^{-1} \boldsymbol{\mu}_x \boldsymbol{\Sigma}_{\epsilon} \right),
       \end{align*}  
       where $\boldsymbol{\Pi} = \mathbb{E}\left[ Z_i\textbf{x}_i\textbf{x}_i^{\top} \right]$.
    \end{lemma}
    \begin{proof}
        Here, we apply the same approach to find the limiting distribution of outcome regression imputed estimator.
        The outcome regression imputed estimator is defined as $\hat{\boldsymbol{\mu}}_{OR} = n^{-1}\SumOverN \hat{\textbf{B}}^{\top}\ObsCov$. Let $\boldsymbol{\theta} = (\boldsymbol{\mu}_y, \vect(\textbf{B}))^{\top}$.
        Define equations for calculating the estimator as:
            \begin{equation}\label{equ:psi4}
                \psi_{1,i}(\boldsymbol{\theta}) = (\IdenMat_{T \times T} \otimes \ObsCov^{\top})\vect(\textbf{B}) - \boldsymbol{\mu}_{y},
            \end{equation}    
            \begin{equation}\label{equ:psi5}
                \psi_{2,i}(\boldsymbol{\theta}) = (\ObsRes - (\IdenMat_{T \times T} \otimes \ObsCov^{\top})\vect(\textbf{B})) \otimes Z_i \ObsCov.
            \end{equation}
        Calculate the Jacobian of $\psi_i(\boldsymbol{\theta})$ and evaluate it at the true value $\boldsymbol{\theta}_0$.
        The first row of Jacobian is
            \begin{equation*}
                \frac{\partial \psi_{1,i}(\boldsymbol{\theta})}{ \partial \boldsymbol{\theta}^{\top} } \biggr\rvert_{\boldsymbol{\theta}_0}  
                = \left[ - \IdenMat_{T \times T}, \IdenMat_{T \times T} \otimes \ObsCov^{\top} \right].
            \end{equation*}        
        The second block of rows is
            \begin{equation*}
                \frac{\partial \psi_{2,i}(\boldsymbol{\theta})}{ \partial \boldsymbol{\theta}^{\top} } \biggr\rvert_{\boldsymbol{\theta}_0} 
                = \left[ \textbf{0}_{Tp \times T},  -\IdenMat_{T \times T} \otimes Z_i\textbf{x}_i\textbf{x}_i^{\top} \right].
            \end{equation*}
        The expected value of Jacobian is 
            \begin{equation*}
                \mathbb{E} \left( \frac{\partial \psi_i(\boldsymbol{\theta})}{ \partial \boldsymbol{\theta}} \biggr\rvert_{\boldsymbol{\theta}_0} \right) 
                    = \begin{bmatrix}
                    -\IdenMat_{T \times T} & \IdenMat_{T \times T} \otimes \boldsymbol{\mu}_x^{\top} \\ 
                    \textbf{0}_{Tp \times T} & -\IdenMat_{T \times T} \otimes \boldsymbol{\Pi}_{p \times p}
                    \end{bmatrix},
            \end{equation*}
        where $\boldsymbol{\Pi} = \mathbb{E}\left[ Z_i \textbf{x}_i\textbf{x}_i^{\top} \right]$. 
        By the M-estimation approach, the asymptotic variance of $\hat{\boldsymbol{\theta}}$ is
            \begin{align*}
                &  \mathbb{E} \left( \frac{\partial \psi_i(\boldsymbol{\theta})}{ \partial \boldsymbol{\theta}} \biggr\rvert_{\boldsymbol{\theta}_0} \right)^{-1} 
                    \mathbb{E} [ \psi_i(\boldsymbol{\theta}_0)\psi_i(\boldsymbol{\theta}_0)^{\top} ]
                    \left\{ \mathbb{E} \left( \frac{\partial \psi_i(\boldsymbol{\theta})}{ \partial \boldsymbol{\theta}} \biggr\rvert_{\boldsymbol{\theta}_0} \right)^{-1} \right\}^{\top} \\
                &= \begin{bmatrix}
                        \textbf{B}^{\top}_0 \boldsymbol{\Sigma}_{x} \textbf{B}_0 + \boldsymbol{\mu}_x^{\top} \boldsymbol{\Pi}^{-1} \boldsymbol{\mu}_x \boldsymbol{\Pi}_{\epsilon} &
                        \boldsymbol{\Sigma}_{\epsilon} \otimes \boldsymbol{\mu}^{\top} \boldsymbol{\Pi}^{-1} \\
                        \boldsymbol{\Sigma}_{\epsilon} \otimes \boldsymbol{\Pi}^{-1} \boldsymbol{\mu}_x & \boldsymbol{\Sigma}_{\epsilon} \otimes \boldsymbol{\Pi}^{-1}
                    \end{bmatrix}.
            \end{align*}
            Then this lemma is complete by the delta method.
    \end{proof}
    Lemma \ref{lemma:finite_DR} states the double robustness property of the DR estimator \cite[e.g.][]{tsiatis2006semiparametric}. 
    However, up to our knowledge, the covariance matrices of the pointwise OR- and DR-estimators in the context of multivariate outcomes do not appear in the literature elsewhere. 
    These covariance matrices are necessary to provide simultaneous inference as presented in Section \ref{sec:feasible}.
\section{Proofs of the tightness lemmas}\label{tight.app} 
    \subsection{Proof of Lemma \ref{lemma:LimRandomElements}}
    \begin{proof}
        Define the $\epsilon$-neighborhood of a subset $\mathcal{S}$ in $\mathcal{H}$ as        
            $$\mathcal{S}^{\epsilon} = \Curlybracket{ f \in \mathcal{H} : \inf \Parentheses{ \Norm{ f - z } : z \in \mathcal{S} } \leq \epsilon}.
            $$
        Let $\Curlybracket{e_j}$ be a complete orthonormal system for the Hilbert space $\mathcal{H}$. 
        Take a finite-dimensional space $\mathcal{S}_J = \Span \Curlybracket{e_1,\dots,e_J}$. 
        Decompose $\xi_n$ as $\xi_{nJ}$ and $\xi_{nJ}'$, the projections on $\mathcal{S}_J$ and the orthogonal complement of $\mathcal{S}_J$, $\mathcal{S}_J^{\perp}$, respectively.
        Then, for any $\epsilon$, if $\xi_n \in \mathcal{S}_J^{\epsilon}$, $\inf\Parentheses{ \Norm{ \xi_n - z} : z \in \mathcal{S}_J } = \Norm{\xi_n - \xi_{nJ}} = \Norm{\xi_{nJ}'} \leq \epsilon$.
        Therefore
            $$
                \Pr\Curlybracket{ \xi_n \in \mathcal{S}_J^{\epsilon} } = \Pr\Curlybracket{ \Norm{\xi_{nJ}'} \leq \epsilon }.
            $$
        Using Chebyshev's inequality 
            \begin{align*}
                \Pr \Curlybracket{ \Norm{\xi_{nJ}'} > \epsilon } &\leq \epsilon^{-2} \mathbb{E} \Parentheses{ \Norm{\xi_{nJ}'}^{2} }.
            \end{align*}
        By independence, zero mean and $\mathbb{E}\Parentheses{\Norm{X_i}^2} < \infty$ assumptions, we have
            \begin{align}\label{eq:ineq}
                \epsilon^{-2} \mathbb{E} \Parentheses{ \Norm{\xi_{nJ}'}^{2} } 
                = \epsilon^{-2} \mathbb{E} \Parentheses{ \Norm{ n^{-1/2}\SumOverN  X_{iJ}'  }^{2} }
                \leq \epsilon^{-2} \max_{i} \Parentheses{\mathbb{E}\Parentheses{ \Norm{X_{iJ}'}^2 }},
           \end{align}        
        where $X_{iJ}'$ is the projection of $X_i$ on $\mathcal{S}_J^{\perp}$.
        Since $\epsilon^{-2}\max_{i} \Parentheses{ \mathbb{E} \Parentheses{ \Norm{X_{iJ}'}^2 } }$ can be arbitrary small if $J$ sufficiently large, $\forall \delta>0$, find a $J_{\epsilon\delta}$, s.t.  
            $$
                \inf_{n\geq 1} \Pr \Parentheses{ \xi_n \in \mathcal{S}_{J_{\epsilon\delta}}^{\epsilon} } \geq 1-\delta.
            $$
        With fixed $J_{\epsilon\delta}$, define $\mathcal{S}^{r}_{J_{\epsilon\delta}} = \Curlybracket{ f \in \mathcal{H}: \max_{j\leq  J_{\epsilon\delta}} |\langle f, e_j \rangle| \leq r}$. We then have
            \begin{align*}
                \Pr \Parentheses{ \xi_n \in \mathcal{S}^{r}_{J_{\epsilon\delta}} } &= 1 - \Pr\Curlybracket{ \max_{j\leq  J_{\epsilon\delta}} |\langle \xi_n, e_j \rangle| > r } \geq 1- \sum_{j}^{J_{\epsilon\delta}} \frac{ \mathbb{E} \langle \xi_n, e_j \rangle ^2 }{r^2} \\
                &= 1 - \frac{ n^{-1}\sum_{j}^{J_{\epsilon\delta}}\SumOverN \mathbb{E} \langle X_i, e_j \rangle ^2 }{r^2} 
                \geq 1 - \frac{ \sum_{j}^{J_{\epsilon\delta}} \max_{i} \Parentheses{\mathbb{E} \langle X_i, e_j \rangle ^2 }}{r^2}.
            \end{align*}
        Let $ r = \sqrt{c/\delta} $ and $c = \sum_{j}^{J_{\epsilon\delta}} \max_{i} \Parentheses{\mathbb{E} \langle X_i, e_j \rangle ^2 }$. Then the proof is complete by Theorem 7.7.4 in \cite{hsing2015theoretical}.
    \end{proof}
    \subsection{Proof of Lemma \ref{lemma:joint}}
    \begin{proof}
        Because of the tightness of $\Curlybracket{\mathcal{P}\circ \Parentheses{X_n | y}^{-1}}_{n\geq1}$, $\forall \epsilon > 0$, $\exists$ compact set $E$ such that $\mathcal{P}\circ \Parentheses{X_n | y}^{-1}\Parentheses{E}\geq 1-\epsilon$.
        By the disintegration representation, we have
        $$
            \mathcal{P}\circ X_n^{-1} \Parentheses{E} = \int_{\mathcal{Y}} \mathcal{P}\circ\Parentheses{X_n|y}^{-1}(E) d \mathcal{P}\circ Y_n^{-1} \geq 1 - \epsilon 
        $$
        Thus $\mathcal{P}\circ X_n^{-1}$ is uniformly tight.
        Similiarly, since $\Curlybracket{\mathcal{P}\circ \Parentheses{X_n | y}^{-1}}_{n\geq1}$ and $\Curlybracket{\mathcal{P}\circ Y_n^{-1}}_{n\geq1}$ are tight, $\forall \epsilon/2 > 0$ and $n$, $\exists$ compact sets $E$ and $F$ s.t. $\mathcal{P}\circ \Parentheses{X_n | y}^{-1}\Parentheses{E}\geq 1-\frac{\epsilon}{2}$ and $\mathcal{P}\circ Y_n^{-1}\Parentheses{F}\geq 1-\frac{\epsilon}{2}$.
        By the disintegration representation, we have
        $$
            \mathbb{P}_n \Parentheses{E \times F} = \int_{F} \mathcal{P}\circ\Parentheses{X_n|y}^{-1}(E) d \mathcal{P}\circ Y_n^{-1} \geq \Parentheses{1-\frac{\epsilon}{2}}^2 > 1 - \epsilon 
        $$
        Then the lemma can be established by clarifying the compactness of $E \times F$. 
        With an arbitrary open cover $\Curlybracket{ \alpha_{\lambda} }_{\lambda \in \Lambda}$ of $E \times F$, for each point $(e,f) \in E \times F$, find a $\lambda_{(e,f)}$ s.t. $(e,f) \in \alpha_{\lambda_{(e,f)}}$. 
        Since $\alpha_{\lambda_{(e,f)}}$ is open, $\exists$ an rectangle $R_{(e,f)} = \mu_{e}^{f} \times \nu_{f}^{e}$, where $\mu_{e}^{f}$ and $\nu_{f}^{e}$ are some open neighbourhoods of $e$ and $f$ respectively, and $R_{(e,f)} \subset \alpha_{\lambda_{(e,f)}}$. 
        For a fixed $e$, $\Curlybracket{\nu_f^e}_{f\in F}$ is an open cover of $F$, and $\exists$ $m$, such that $F \subset \bigcup_{j = 1}^m \nu_{f_j}^e$ by compactness.  
        Since $\Curlybracket{\bigcap_{j = 1}^m \mu_e^{f_j}}_{e\in E}$ is also an open cover of $E$, again, by compactness, $\exists$ $n$ such that $E \subset \bigcup_{i = 1}^n \Parentheses{\bigcap_{j = 1}^m \mu_{e_i}^{f_j}}$.
        Then, $E \times F$ is covered by $\bigcup_{i = 1}^n\bigcup_{j = 1}^m  \Parentheses{\bigcap_{j = 1}^m \mu_{e_i}^{f_j}} \times \nu_{f_j}^{e_i}  \subset \bigcup_{i = 1}^n\bigcup_{j = 1}^m \mu_{e_i}^{f_j} \times \nu_{f_j}^{e_i}\subset \bigcup_{i = 1}^n\bigcup_{j = 1}^m R_{(e_i, f_j)}$.
        Thus we can find a finite number of sets in $\Curlybracket{ \alpha_{\lambda} }_{\lambda \in \Lambda}$ to cover $E \times F$, and $E \times F$ is compact. 
    \end{proof}
\section{Further results from the Monte Carlo study}\label{figure.app}
    We include below simulation results on the OR- and DR-estimators mentioned in the main text, including bias, mean estimated variance as well as Monte Carlo variance, and MSE (combining mean squared bias and Monte Carlo variance).
    We observe low bias when at least one model is correctly specified.  
    Bias and Monte Carlo variability decrease with sample sizes, and estimated variances correspond to the Monte Carlo variants. 
    Furthermore, Figure \ref{fig:bandwidth} shows the average width of the simultaneous and pointwise confidence bands in one specific scenario and at two different sample sizes. As would be expected, the pointwise bands are notably narrower. 

       \begin{figure}[t!]
        \centering
        \includegraphics[width=\linewidth]{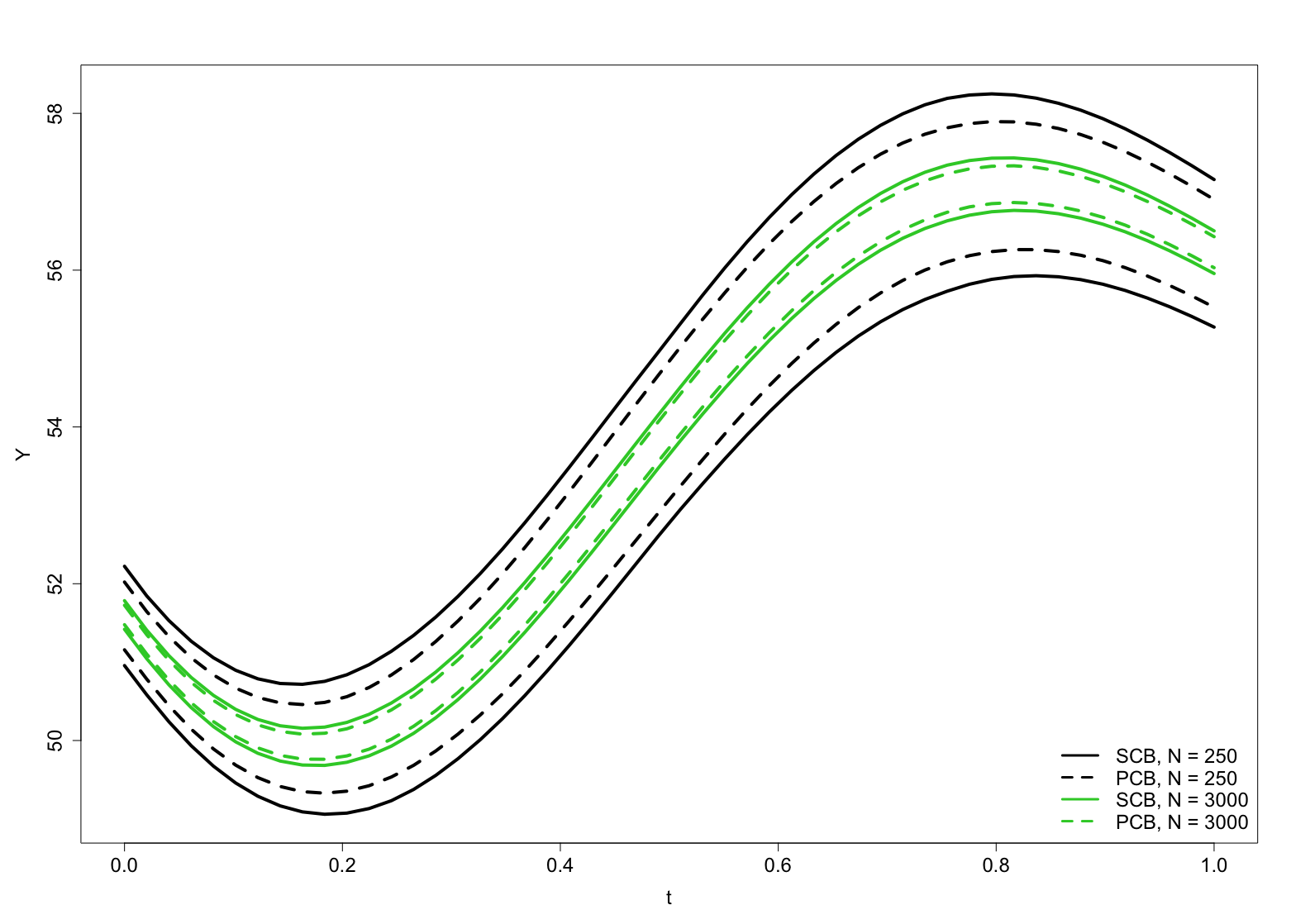}
        \caption{Average SCB (solid) and PCB (dashed) from 1000 simulation replicates at n = 250 (black) and n = 3000 (green). Based on MVN errors for the OR model with no model misspecification.}
        \label{fig:bandwidth}
    \end{figure}
    
    \begin{figure}[htbp]
        \centering
        \begin{minipage}[t]{ 0.8\linewidth}
            \centering
            \includegraphics[width=\linewidth]{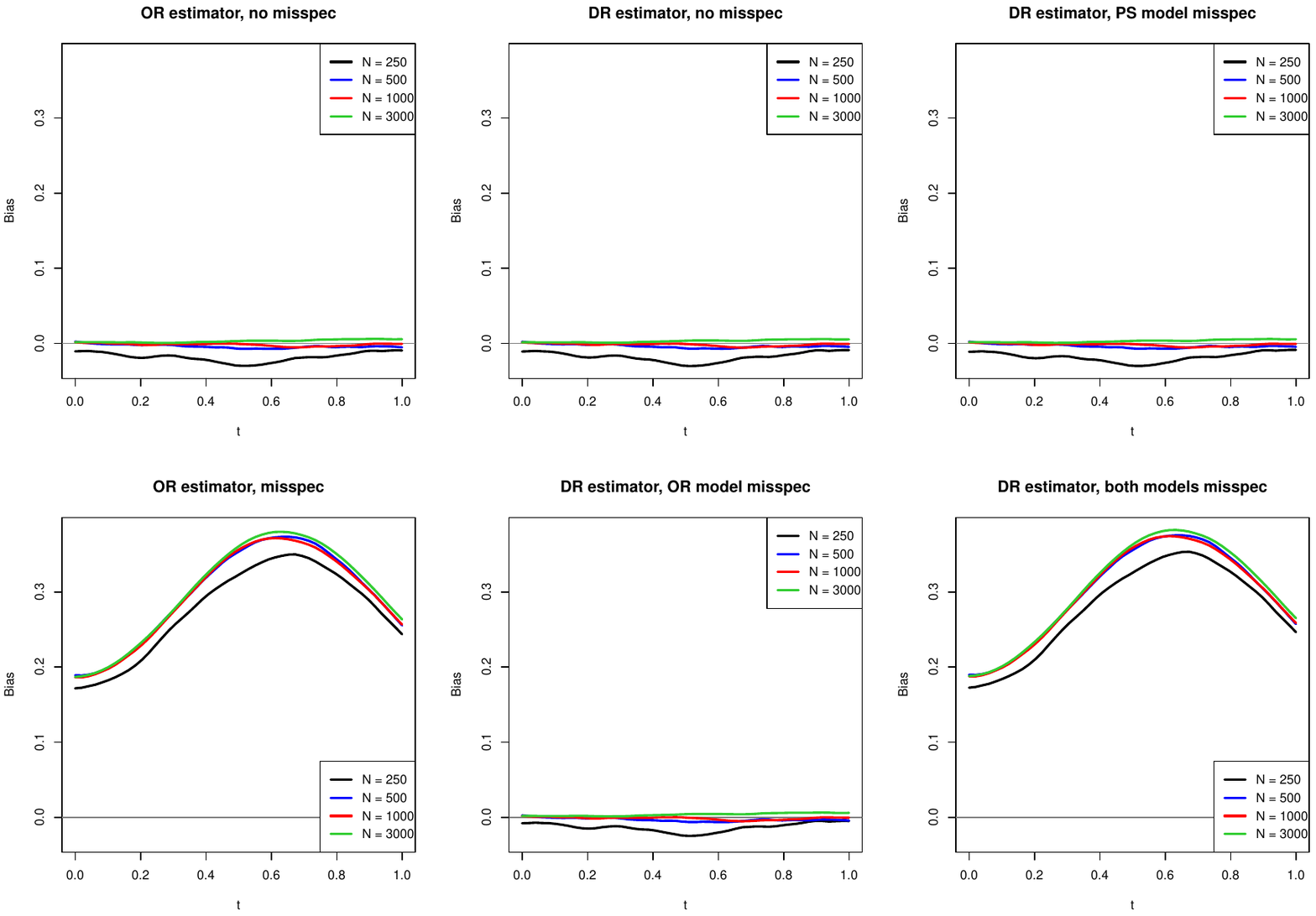}
            \caption{\footnotesize Bias for the Monte Carlo simulations using Gaussian error terms for the OR estimator (top and bottom left panels) and DR estimator (middle and right panels); OR and PS models are outcome and propensity score models. By comparison, the bias of the complete case estimate varies between 0.21 and 1.39.}
        \label{fig:bias_gp}
        \end{minipage}
        \begin{minipage}[t]{ 0.8\linewidth}
            \centering
            \includegraphics[width=\linewidth]{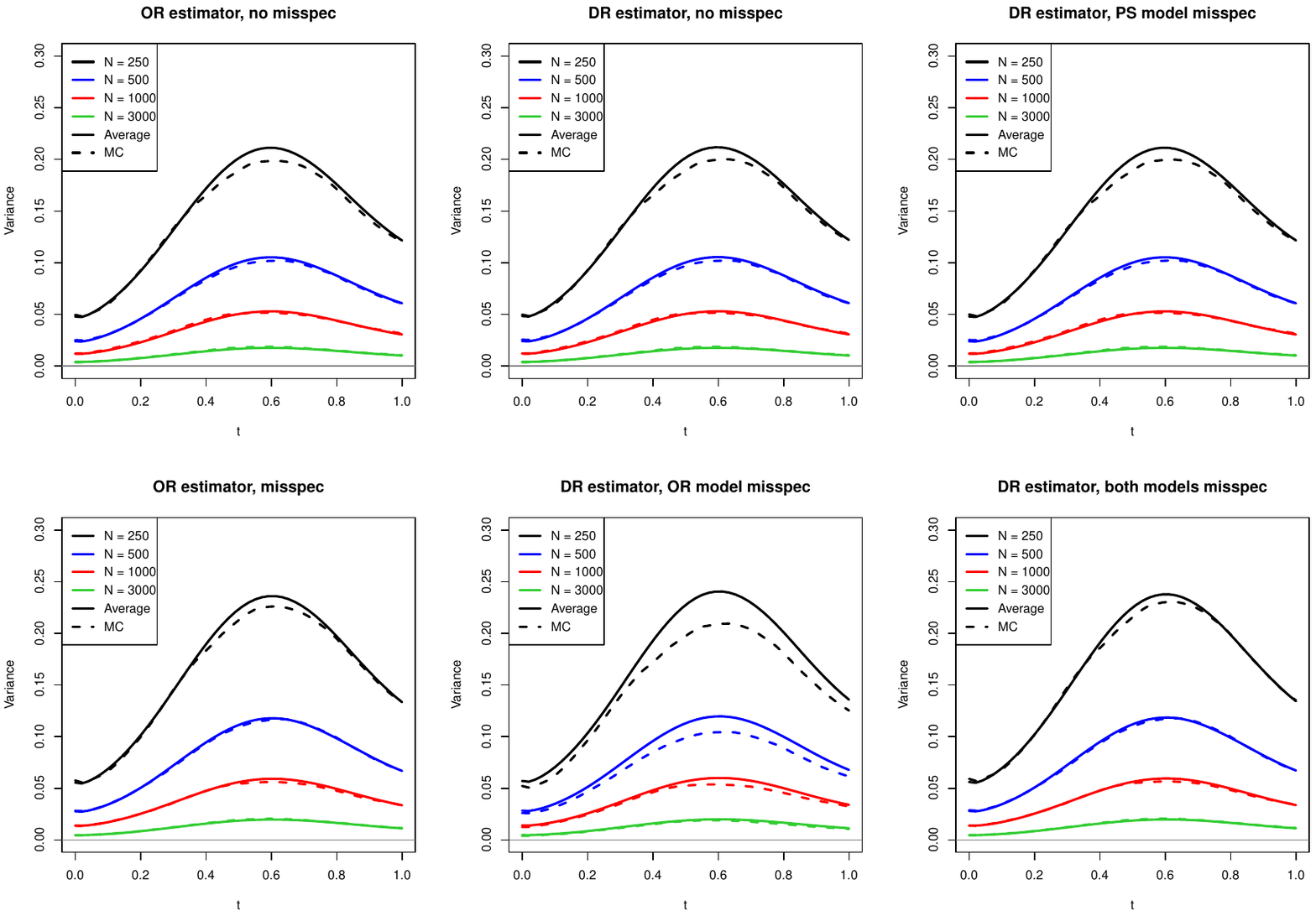}
            \caption{\footnotesize Mean estimated variances (solid lines) and Monte Carlo (MC) variance (dashed lines) for the Monte Carlo simulations using Gaussian error terms. Results for the OR estimator (top and bottom left panels) and DR estimator (middle and right panels); OR and PS models are outcome and propensity score models. By comparison, the results for the complete case estimate are at similar levels.}
        \label{fig:variance_gp}
        \end{minipage}
    \end{figure}

    \begin{figure}[htbp]
        \centering
        \begin{minipage}[t]{ 0.8\linewidth}
            \centering
            \includegraphics[width=\linewidth]{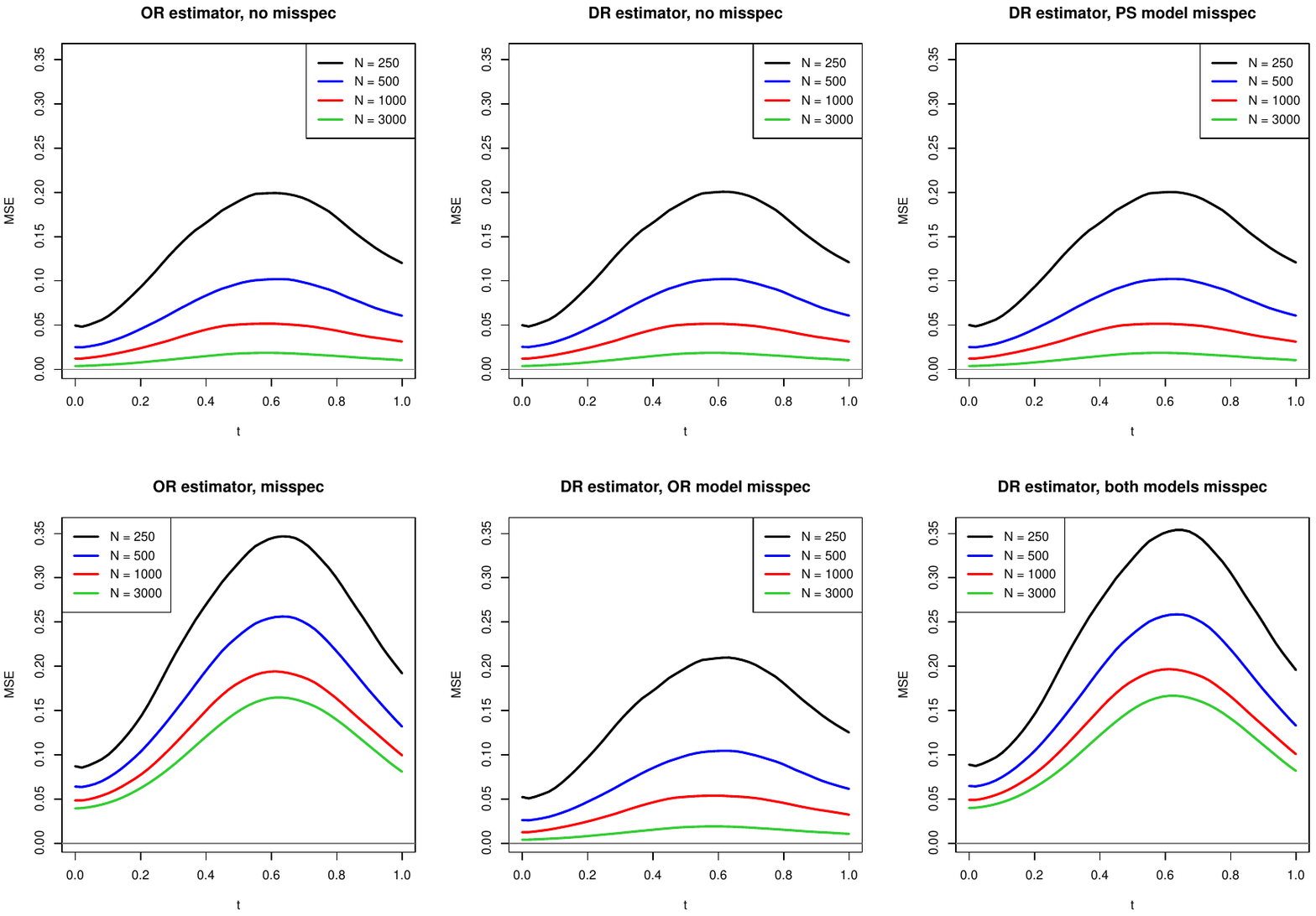}
            \caption{\footnotesize MSE for the Monte Carlo simulations using Gaussian error terms for the OR estimator (top and bottom left panels) and DR estimator (middle and right panels); OR and PS models are outcome and propensity score models. By comparison, the MSE of the complete case estimate varies between 0.06 and 2.12.}
        \label{fig:mse_gp}
        \end{minipage}
        \begin{minipage}[t]{ 0.8\linewidth}
            \centering
            \includegraphics[width=\linewidth]{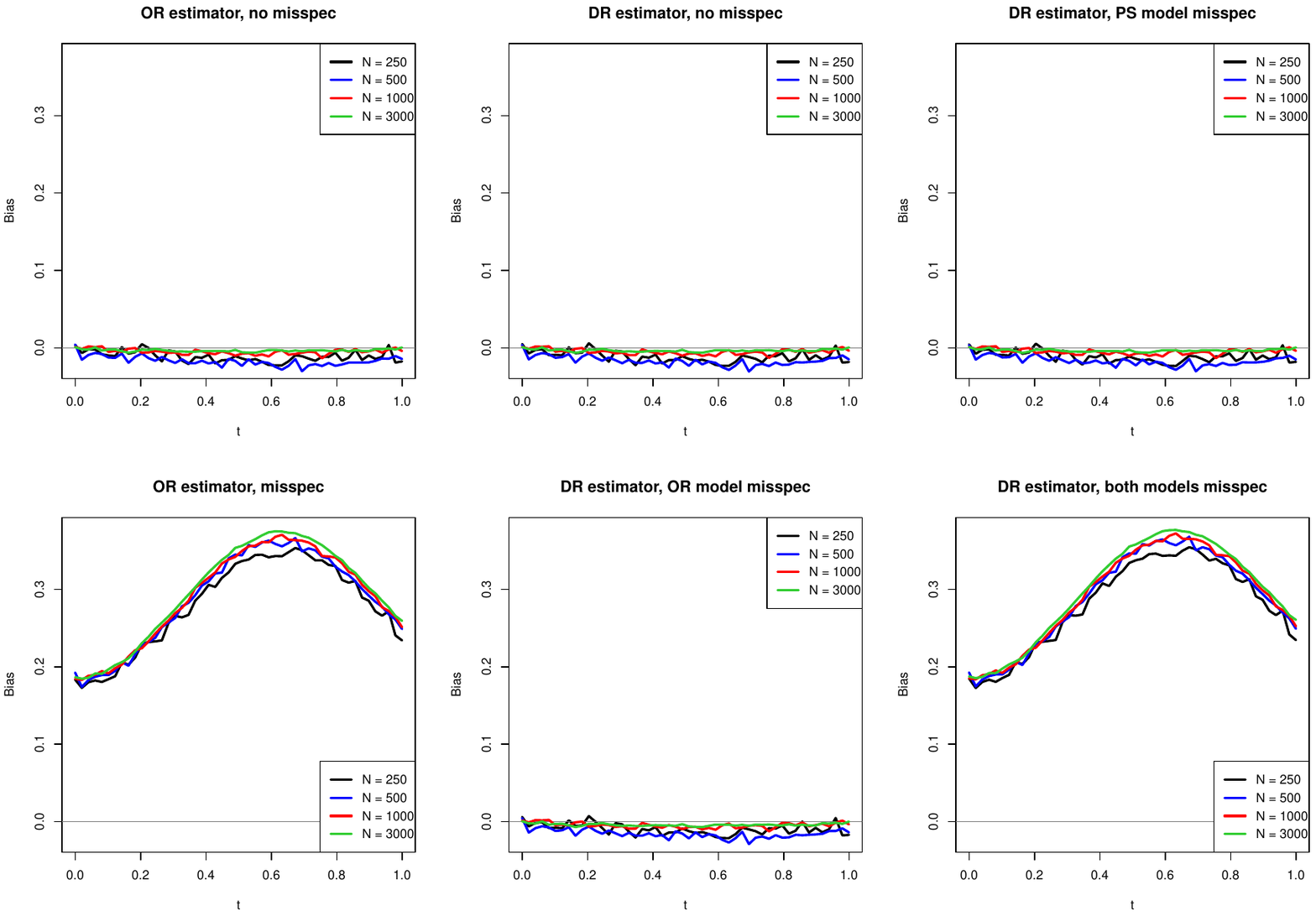}
            \caption{\footnotesize Bias for the Monte Carlo simulations using multivariate t-distributed error terms for the OR estimator (top and bottom left panels) and DR estimator (middle and right panels); OR and PS models are outcome and propensity score models.By comparison, the bias of the complete case estimate varies between 0.22 and 1.39.}
        \label{fig:bias_t}
        \end{minipage}
    \end{figure}    

    \begin{figure}[htbp]
        \centering
        \begin{minipage}[t]{ 0.8\linewidth}
            \centering
            \includegraphics[width=\linewidth]{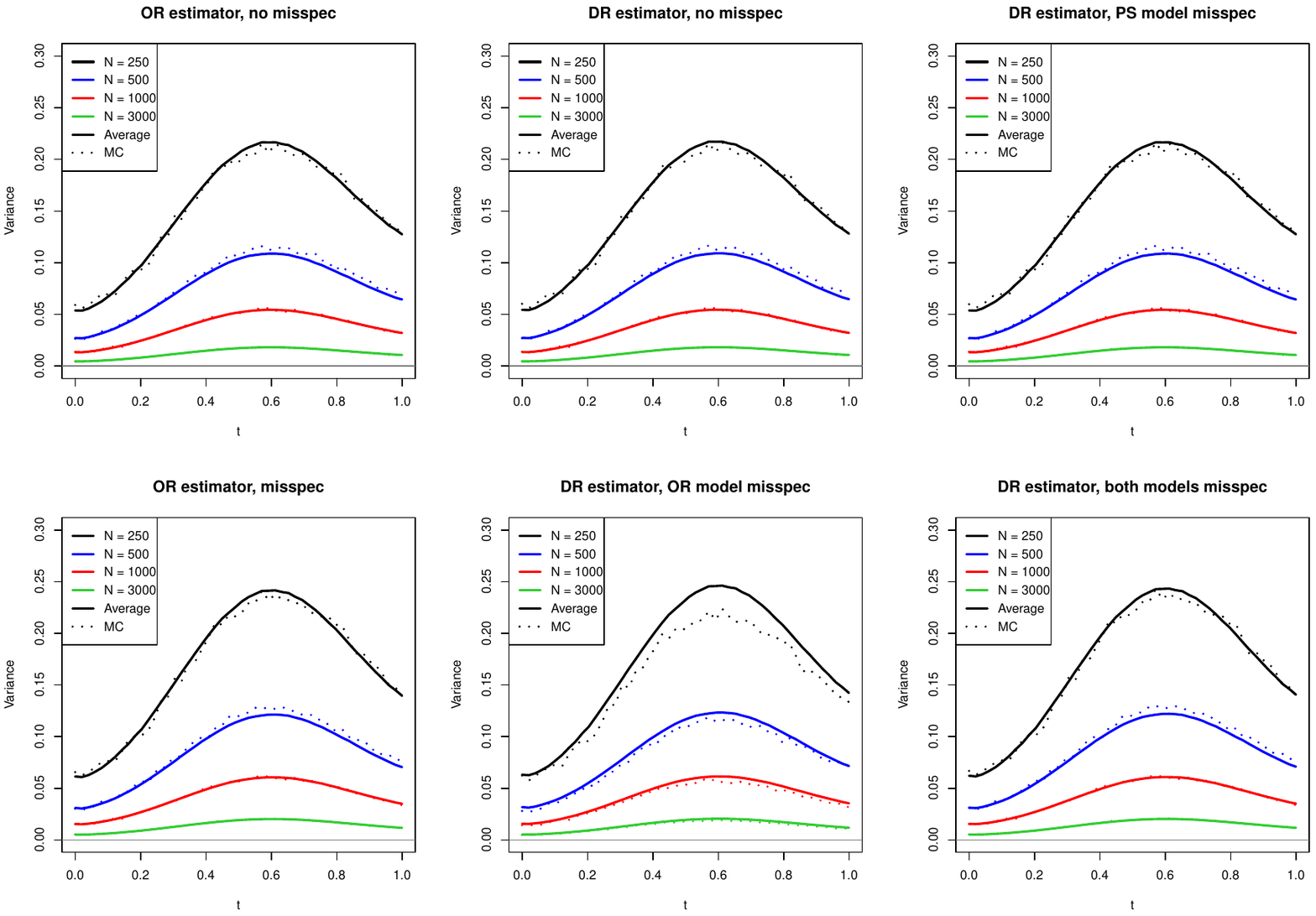}
            \caption{\footnotesize Mean estimated variances (solid lines) and Monte Carlo (MC) variance (dotted lines) for the Monte Carlo simulations using multivariate t-distributed error terms. Results for the OR estimator (top and bottom left panels) and DR estimator (middle and right panels); OR and PS models are outcome and propensity score models. By comparison, the results for the complete case estimate are at similar levels.}
        \label{fig:variance_t}
        \end{minipage}
        \begin{minipage}[t]{ 0.8\linewidth}
            \centering
            \includegraphics[width=\linewidth]{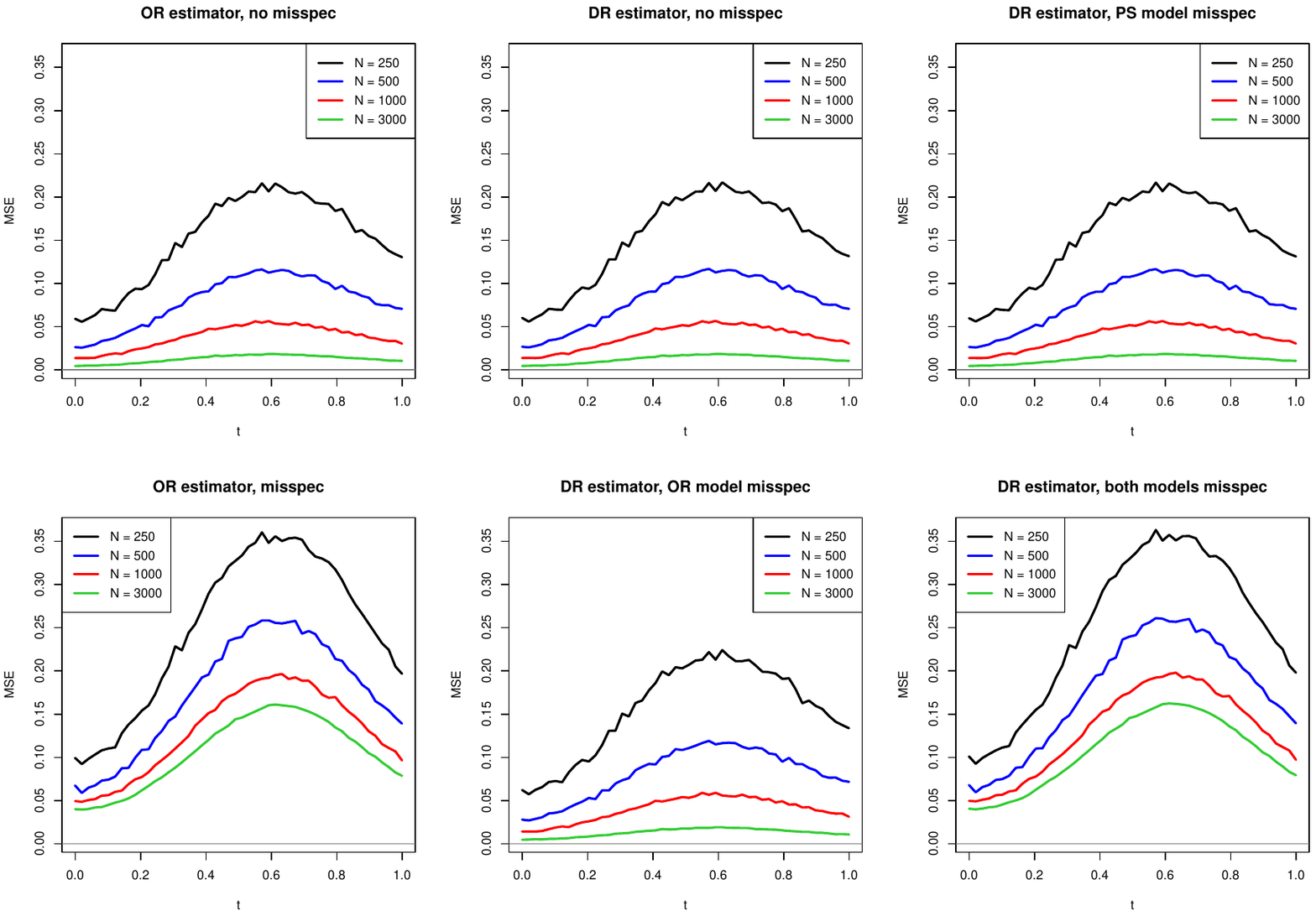}
            \caption{\footnotesize MSE for the Monte Carlo simulations using multivariate t-distributed error terms for the OR estimator (top and bottom left panels) and DR estimator (middle and right panels); OR and PS models are outcome and propensity score models. By comparison, the MSE of the complete case estimate varies between 0.05 and 2.12.}
        \label{fig:mse_t}
        \end{minipage}
    \end{figure} 
    
\end{document}